\documentclass[hidelinks, 11pt]{article}
\usepackage[a4paper,bindingoffset=0.in,left=2.5cm,right=2.5cm,top=2cm,bottom=3cm]{geometry}
\usepackage{amsmath, amssymb,amsthm}
\usepackage{graphicx,subfigure}
\usepackage{url}
\usepackage{enumitem}
\usepackage[utf8]{inputenc}
\usepackage{lmodern}
\usepackage{setspace}
\usepackage[normalem]{ulem}

\usepackage{wrapfig}
\usepackage{mdframed}
\usepackage[a]{esvect}
\usepackage{color}

\newtheorem{theorem}{Theorem}[section]
\newtheorem{proposition}[theorem]{Proposition}

\newtheorem{lemma}[theorem]{Lemma}

\newtheorem{corollary}[theorem]{Corollary}

\theoremstyle{definition} % This theoremstyle is non-italic
\newtheorem{definition}[theorem]{Definition}
\newtheorem{remark}[theorem]{Remark}
\newtheorem{example}[theorem]{Example}

\usepackage{mathtools, amsmath, amssymb, graphicx, amsfonts}
\usepackage{xcolor} %to color comments
\usepackage[colorlinks=true,allcolors=blue!75!black,backref=page]{hyperref} %to color links and to provide refs to the pages where something was referenced
\usepackage{bbm} %to use indicator function
\usepackage{float} %to use [h!] with figures
\usepackage{amsopn} %standard?
\usepackage{mathrsfs} %for \mathscr
\usepackage{comment} %to comment out sections
\usepackage{hyperref} %for hyperlink
\usepackage{amsthm} %for proof environment
\renewenvironment{proof}{{\noindent\bfseries Proof.~}}{\qed} %to get bf proof
\usepackage{thmtools} %to recall a theorem
\usepackage{thm-restate} %to recall a theorem

\newcommand{\im}{\operatorname{im}}
\newcommand{\diag}{\operatorname{diag}}

\newcommand{\rank}{\operatorname{rank}}
\newcommand{\conv}{\operatorname{conv}}

\newcommand{\Gr}{\operatorname{Gr}}
\newcommand{\row}{\operatorname{row}}
\DeclareMathOperator{\RQ}{RQ}
\DeclareMathOperator{\vol}{vol}

\usepackage[affil-it]{authblk}

\title{Uniform density in matroids, matrices and graphs}

\author[1]{Karel Devriendt}

\author[2]{Raffaella Mulas}
\affil[1]{Max Planck Institute for Mathematics in the Sciences, Leipzig, Germany}
\affil[2]{Vrije Universiteit Amsterdam, Amsterdam, The Netherlands}
\date{}

\begin{document}

%%%%%%%%%%
% ABSTRACT
%%%%%%%%%%
\maketitle 
\begin{abstract}
We give new characterizations for the class of uniformly dense matroids and study applications of these characterizations to graphic and real representable matroids. We show that a matroid is uniformly dense if and only if its base polytope contains a point with constant coordinates. As a main application, we derive new spectral, structural and classification results for uniformly dense graphs. In particular, we show that connected regular uniformly dense graphs are $1$-tough and thus contain a (near-)perfect matching. As a second application, we show that strictly uniformly dense real represented matroids can be represented by projection matrices with a constant diagonal and that they are parametrized by a subvariety of the Grassmannian. 
\end{abstract}
 
%%%%%%%%%%
%  BODY  %
%%%%%%%%%%
\section{Introduction}
Matroids were introduced by Whitney and Nakasawa as an abstraction and generalization of independence in linear algebra (linear independence) and graph theory (acyclicity), and they have a rich combinatorial, geometric and representation theory; a standard reference is \cite{oxley_2011_matroid}. A \emph{matroid}\footnote{ 
We note that this is one of many alternative ways, so-called cryptomorphisms, to define matroids. We refer to \cite{oxley_2011_matroid} for an overview of the most common definitions.} $M=(E,\mathcal{B})$ is a nonempty collection $\mathcal{B}$ of subsets of a ground set $E$ that satisfy:
\begin{equation}\label{eq: base-exchange axiom}
\text{for all~} A,B\in \mathcal{B} \text{~and all $a\in A\backslash B$, there exists $b\in B\backslash A$ such that $(A\backslash \lbrace a\rbrace)\cup\lbrace b\rbrace \in \mathcal{B}$}.
\end{equation}
Elements of $\mathcal{B}$ are called the \emph{bases} of the matroid and an important consequence of property \eqref{eq: base-exchange axiom} is that every basis has the same cardinality, which is called the \emph{rank of the matroid}. The \emph{rank} can be extended to a function on subsets of the ground set as
$$
\rank(A) := \max\lbrace \vert B\cap A\vert \,:\, \text{$B$ is a basis}\rbrace, \quad\text{for all $A\subseteq E$}.
$$
Following the terminology of \cite{catlin_1992_arboricity}, we define the \emph{density} of a subset of the ground set as
$$
\rho(A)\,:=\, \frac{\vert{A}\vert}{\rank(A)}, \quad\text{~for all $A\subseteq E$},
$$
and set by definition $\rho(A)=\infty$ when $\rank(A)=0$, and $\rho(\emptyset)=1$. In this article, we will often work with \emph{loopless matroids}, which have $\rho(A)<\infty$ for all subsets (see Theorem \ref{th: no loops or coloops}). We write $\rho=\rho(M):=\rho(E)$, and call this the {density} of the matroid. A matroid is called \emph{uniformly dense} if $\rho(A)\leq \rho(E)$ for all subsets $A\subseteq E$. 
\\~\\
Narayanan and Vartak \cite{narayanan_1981_elementary} introduced uniformly dense matroids in the context of certain electrically-inspired partitions of graphs and matroids, now known as principal partitions \cite{fujishige_2009_theory}, and studied their basic properties in \cite{narayanan_1981_molecular}. Catlin, Grossman, Hobbs and Lai \cite{catlin_1992_arboricity} introduced the name ``uniformly dense matroids'' and showed that they are closely related to {fractional arboricity} --- this is the smallest $\alpha$ such that any collection of bases that covers every element of the ground set at least $t$ times, has size at least $\alpha\times t$; see \cite{payan_1986_graphes} --- and {strength} --- this is the largest $\beta$ such that the rank of any subset of size $\vert E\vert-k$ is at least $\rank(E)-k/\beta$; see \cite{gusfield_1983_connectivity, cunningham_1985_optimal}. Catlin et al.\ showed that uniformly dense matroids are precisely those matroids for which strength and fractional arboricity coincide. Independently, uniformly dense matroids appeared in a conjecture on orderings of the bases of a matroid \cite{Kajitani_1988_ordering}; the conjecture says that a matroid $M$ is uniformly dense if and only if there exists an ordering of $E$ such that all $\rank(M)$ cyclically consecutive elements form a basis. Van den Heuvel and Thomass\'{e}e \cite{vandenheuvel_2012_cyclic} proved this conjecture for matroids whose size and rank are coprime, McGuinness proved that it holds for all paving matroids \cite{mcguiness_2024_cyclic}, and B\`{e}rczi, J\'{a}nosik and M\'{a}trav\"{o}lgyi proved that it holds for all split matroids \cite{berczi_2024_cyclic}. One further context in which uniform density appears is as a combinatorial encoding of (semi-)stability, as it appears in geometric invariant theory. As pointed out in \cite{geiger_2022_self-dual}, a collection of points in projective space is stable (respectively semistable) if and only if the corresponding representable matroid is uniformly dense (respectively strictly uniformly dense); see also \cite[Ch. II, Thm. 1]{dolgachev_1989_point} for details.
\\
~
\\
The analysis of uniformly dense matroids in this article starts from the observation that the density inequalities $\rho(A)\leq \rho(E)$ can be interpreted geometrically as the hyperplane inequalities of a certain polyhedron, the base polytope $P(M)$ associated with the matroid $M$
\begin{align}
P(M) &:= \conv\left(\left\lbrace e_B\,:\,\text{$B$ is a basis}\right\rbrace\right)\subset \mathbb{R}^{\vert E\vert} \nonumber
\\
&= \left\lbrace x\in\Delta_M \,:\, \sum_{e\in A}x_e \leq \rank(A) \text{~for all $\emptyset\neq A\subseteq E$}\right\rbrace,\label{eq: definition matroid polytope}
\end{align}
where $e_B$ is a $0$--$1$ indicator vector with entries $1$ for the elements in $B$, and with the hypersimplex $\Delta_M:=\left\lbrace x\in\mathbb{R}^{\vert E\vert}_{\geq 0} \,:\, \sum_e x_e=\rank(M)\right\rbrace$. Matroid polytopes were introduced by Edmonds \cite{edmonds_1970_submodular} in the context of combinatorial optimization and independently by Gelfand, Goresky, MacPherson and Serganova \cite{gelfand_1987_combinatorial} while studying cell decompositions of the Grassmannian. The halfspace description of $P(M)$ in \eqref{eq: definition matroid polytope} can be found in \cite{feichtner_2005_matroid}. Our new characterization of uniformly dense matroids follows immediately from the definition of the matroid polytope.
\begin{restatable}{theorem}{characterization}\label{th: main characterization theorem}
The following are equivalent for a matroid $M$:
\begin{enumerate}
    \item[(1)] $P(M)$ contains the point $(\rho^{-1},\dots,\rho^{-1})$.

    \item[(2)] $\rho(A)\leq \rho(E)$ for all subsets $A \subseteq E$.

    \item[(3)] There exists a measure $\mu$ on $\mathcal{B}$ such that $\mu(\lbrace B : B\ni e\rbrace)$ is equal for all $e\in E$.
\end{enumerate}
A matroid that satisfies these conditions is called \emph{uniformly dense}.
\end{restatable}
\begin{proof}
$(1)\Leftrightarrow (2).$ The matroid polytope contains the point $(\rho^{-1},\dots,\rho^{-1})$ if and only if this point satisfies the hyperplane inequalities $\sum_{e\in A}\rho^{-1} \leq \rank(A)$, for all nonempty $A$. Equivalently, 
$$
\vert A\vert \cdot \rho^{-1} \leq \rank(A) \iff \rho(A)=\frac{\vert A\vert}{\rank(A)} \leq \rho =\rho(E).
$$

$(1) \Leftrightarrow (3).$ The matroid polytope contains the point $(\rho^{-1},\dots,\rho^{-1})$ if and only if this point is a convex combination of the vertices of $P$. Equivalently, there must exist nonnegative coefficients $\mu(B)$ that add up to one, such that $$\sum_{B\in\mathcal{B}}\mu(B) e_B = (\rho^{-1},\dots,\rho^{-1})$$ or, for each coordinate $e\in E$, 
$$
\sum_{B:B\ni e}\mu(B) = \rho^{-1}.
$$ 
These coefficients are the desired measure $\mu$ on $\mathcal{B}$. Since any measure can be normalized by dividing by $\mu(\mathcal{B})$, normalization is not necessary for the measure $\mu$ in condition \textit{(3)}.
\end{proof}

If the point $(\rho^{-1},\dots,\rho^{-1})$ in \textit{(1)} lies in the relative interior of $P(M)$, then the matroid is called \emph{strictly uniformly dense}. A measure that satisfies condition \textit{(3)} in Theorem \ref{th: main characterization theorem} will be called an \textit{$E$-uniform basis measure}. The normalized $E$-uniform basis measures associated with a uniformly dense matroid form a polytope: this is the intersection of the affine space 
$$
\left\{ \mu\in\mathbb{R}^{\vert\mathcal{B}\vert}\,:\,\sum_{B\ni e}\mu(B)=\rho^{-1}, ~\forall e\in E\right\}
$$ with the probability simplex. Several known results on uniformly dense matroids, for instance those in \cite{catlin_1992_arboricity} related to fractional covering and packing, the conjecture of Kajitani et al.\ in \cite{Kajitani_1988_ordering} as well as our Theorem \ref{th: real UD matroids have determinantal measures}, can be interpreted as statements about certain points in this polytope.

As a first application of Theorem \ref{th: main characterization theorem} we study graphic matroids, where $E$ are the edges of a simple graph $G$ and bases are spanning forests. The density of this matroid is denoted by $\rho(G)$ and a graph is said to be uniformly dense if the corresponding graphic matroid is uniformly dense. We study some implications of uniform density on the Laplacian spectrum of graphs and classify all connected uniformly dense graphic matroids with two cycles. Using $E$-uniform basis measures, we prove a connectivity result related to toughness: a graph $G$ is \emph{$t$-tough} if removing any $k$ vertices yields a graph with at most $k/t$ components \cite{bauer_2006_toughness}. Any graph with maximum degree $\Delta(G)$ is $(1/\Delta(G))$-tough, and for uniformly dense graphs we find the following improvement:
\begin{restatable*}{theorem}{toughness}\label{th: toughness of uniformly dense graphs}
A uniformly dense graph $G$ is $(\rho(G)/\Delta(G))$-tough. If, furthermore, $G$ is regular and connected, then it is $1$-tough. 
\end{restatable*}
We note that these types of vertex-based connectivity results are new in the context of uniform density, as compared to the existing edge-based connectivity results. This result also implies the existence of (near-)perfect matchings in regular uniformly dense graphs. The proof of Theorem \ref{th: toughness of uniformly dense graphs} is based on techniques of Fiedler \cite{fiedler_2011_matrices} and is given in Section \ref{subsection: connectivity for ud graphs}.

As a second application of Theorem \ref{th: main characterization theorem} we consider real representable matroids, where $E$ is the set of columns of a real matrix and bases are maximal linearly independent sets of columns. For these matroids, we show in Theorem \ref{th: real UD matroids have determinantal measures} that there exists an $E$-uniform measure of the form $\mu_X:B\mapsto\det(X_B)^2,$ given by the minors of some matrix $X$.

A projection representation of a matroid is a projection matrix whose non-singular principal submatrices determine the bases of the matroid. Using results from the theory of determinantal point processes (see e.g.\ \cite{lyons_2003_determinantal}) we show the following result in Section \ref{subsection: SUD real representable}:

\begin{restatable*}{theorem}{projectionRepresentation}\label{th: representable UD matroids and projection matrices}
Let $M$ be a real representable matroid. Then $M$ is strictly uniformly dense if and only if it has a projection representation with constant diagonal equal to $\rho(M)^{-1}$.
\end{restatable*}
As a final application, we show that strictly uniformly dense real representable matroids of size $n$ and rank $k$ can be parametrized by a subvariety $\mathcal{V}(n,k)$ of the Grassmannian, which parametrizes real representable matroids. This subvariety is further studied in follow-up work \cite{devriendt_2024_grassmannian}.
\\~\\
\textbf{Paper outline.} Section \ref{section: introduction uniformly dense matroid} recalls the definition of uniformly dense matroids, together with some examples. We describe basic properties and survey some relevant known results. 

Section \ref{section: graphic matroids} deals with uniformly dense graphic matroids. We introduce the required terminology and concepts from graph theory and give some examples. In Section \ref{subsection: structure of ud graphs} we study implications of uniform density on the structure of a graph and classify uniformly dense graphs with two cycles. Section \ref{subsection: connectivity for ud graphs} discusses connectivity of uniformly dense graphs with the toughness result Theorem \ref{th: toughness of uniformly dense graphs} as the main result. Section \ref{subsection: spectral results for ud graphs} finally describes results on the normalized Laplacian spectrum of uniformly dense graphs. 

Section \ref{section: real matroids} deals with real representable matroids. In Section \ref{subsection: SUD real representable} we consider strictly uniformly dense real representable matroids. We show that these matroids admit an $E$-uniform `determinantal' measure in Theorem \ref{th: real UD matroids have determinantal measures} and that they can be represented by a constant-diagonal projection matrix in Theorem \ref{th: representable UD matroids and projection matrices}. Section \ref{subsection: moduli space} introduces a variety that parametrizes strictly uniformly dense real matroids.

%%%%%%%%%%%%%%%%%%%%%%%%%%%%%%%%
%                              %
%%%%%%%%%%%%%%%%%%%%%%%%%%%%%%%%
\section{Uniformly dense matroids}\label{section: introduction uniformly dense matroid}
\subsection{Definition and examples}
As introduced, a matroid $M=(E,\mathcal{B})$ consists of a nonempty set $\mathcal{B}$ that contains subsets of a finite ground set $E$, such that the base-exchange property \eqref{eq: base-exchange axiom} is satisfied. For convenience, we will often focus on matroids where every element of the ground set is contained in at least one element of $\mathcal{B}$; these are called loopless matroids. 

The elements $B\in\mathcal{B}$ are called the bases of $M$ and a fundamental consequence of the base-exchange property is that every basis has the same cardinality. This is called the \emph{rank} of $M$, and the rank of a set $A\subseteq E$ is defined as the size of the largest intersection of $A$ with any basis. A standard reference on matroid theory is \cite{oxley_2011_matroid}. We now define the ingredients for the definition of uniformly dense matroids.

\textbf{Density:} Narayanan and Vartak \cite{narayanan_1981_elementary} defined the density of a subset $A\subseteq E$ as $\rho(A):=\vert A\vert/\rank(A)$ and the density of a matroid as $\rho(M):=\rho(E)$. In this definition, we set $\rho(A)=\infty$ when $\rank(A)=0$ and set $\rho(\emptyset)=1$. The possible densities of a matroid are $\rho(M)\in\lbrace \vert E\vert/1,\dots,\vert E\vert/\vert E\vert\rbrace \cup\{\infty\}$ and Example \ref{ex: uniform matroids} shows that every density is realized by some matroid.

\textbf{Basis measure:} A measure $\mu$ on the bases of $M$ is determined by a nonnegative real number $\mu(B)$ for each basis $B$. The measure is additive and we write $\mu(\mathcal{A})=\sum_{B\in\mathcal{A}}\mu(B)$ for the measure of a set $\mathcal{A}\subseteq\mathcal{B}$ of bases. A measure $\mu$ is \emph{normalized} if $\mu(\mathcal{B})=1$, but this is not required in general. We note the following useful lemma on basis measures.
\begin{lemma}\label{lemma: foster euler marginal sum}
Every measure $\mu$ on $\mathcal{B}$ satisfies $$\sum_{e\in E}\mu(\lbrace B : B\ni e\rbrace)=\rank(M)\mu(\mathcal{B}).$$
\end{lemma}
\begin{proof}
$$
\sum_{e\in E}\mu(\lbrace B : B\ni e\rbrace) = \sum_{e\in E}\sum_{B:B\ni e}\mu(B) = \sum_{B\in\mathcal{B}}\mu(B)\vert B\vert = \rank(M)\mu(\mathcal{B}).
$$
\end{proof}

\textbf{Base polytope:} The base polytope $P(M)$ of a matroid is defined as the convex hull of indicator vectors of its bases $P(M)=\conv\left(\lbrace e_B \,:\, B\text{~is a basis}\rbrace\right)\subset\mathbb{R}^{\vert E\vert}.$ See Equation \eqref{eq: definition matroid polytope} or reference \cite{feichtner_2005_matroid} for the description of this polytope in terms of hyperplane inequalities. 

The following theorem from the introduction may be used as the definition of uniformly dense matroids.
\characterization*

If in \textit{(1)} the point lies in the relative interior of $P(M)$ or in \textit{(3)} the measure has full support, then we say that $M$ is \emph{strictly uniformly dense}. The second condition is a bit more subtle and is dealt with in Theorem \ref{theorem: rank conditions for strict uniform density} at the end of Section \ref{subsection: basic structure and operations}. We now give some examples of matroids and uniformly dense matroids.

\begin{example}[Tadpole matroid]\label{ex: tadpole matroid}
The tadpole matroid $T_4$ is the matroid with ground set $[4]$ and bases 123, 134, 124, where strings denote subsets. The name ``tadpole'' is clarified in Example \ref{ex: tadpole graph} and Figure \ref{fig: examples of uniformly dense graphs}. Since $\rho(234)=3/2 > 4/3 = \rho(1234)$, the tadpole matroid is not uniformly dense. Theorem \ref{th: no loops or coloops} will show that the fact that element $1$ is in every basis of $T_4$ is an obstruction for uniform density. The matroid polytope $P(T_4)$ is the convex hull of the points $$\lbrace(1,1,1,0),(1,1,0,1),(1,0,1,1)\rbrace\subset \mathbb{R}^4.$$ This is a regular triangle in the affine hyperplane $\lbrace x\in\mathbb{R}^4:x_1=1\rbrace$ and does not contain the point $(4/3,\dots,4/3)$.
\end{example}

\begin{example}[Uniform matroid]\label{ex: uniform matroids}
The \emph{uniform matroid} $U_{k,n}$ of size $n$ and rank $k$ is the matroid with ground set $E=[n]:=\lbrace 1,\dots,n\rbrace$ and bases $\mathcal{B}={[n]\choose k}:=\lbrace B\subseteq [n] : \vert B\vert=k\rbrace$ given by all $k$-element subsets of $[n]$. Since $\rho(U_{k,n})=n/k$, there exists a matroid with any rational density of at least 1. We note in particular that $\rho(U_{0,n})=\infty$ when $n>0$, and $\rho(U_{0,0})=1$.

Let $\mu:B\mapsto 1$ be the uniform measure on the bases of $U_{k,n}$, with $k>0$. Then $\mu(\lbrace B: B\ni e\rbrace)={n-1 \choose k-1}$ is independent of $e\in E$, which proves that these uniform matroids are uniformly dense. The matroid polytope of a uniform matroid is equal to the associated hypersimplex $P(U_{k,n})=\Delta_{U_{k,n}}$ and contains the point $(k/n,\dots,k/n)$. The rank zero uniform matroids $U_{0,n}$ and $U_{0,0}$ satisfy the conditions of Theorem \ref{th: main characterization theorem} either vacuously or trivially, and thus are uniformly dense.
\end{example}

\begin{example}[Self-dual matroids]
A matroid $M$ is \emph{self-dual} if $B\in\mathcal{B}\Leftrightarrow E\backslash B\in\mathcal{B}.$ In particular, self-dual matroids have density $\rho(M)=2$. For any fixed basis $A$, the basis measure 
$$\mu_A:B\mapsto\lbrace 1/2 \text{~if $B=A$~or~$B=E\backslash A$, and $0$ otherwise}\rbrace$$ is $E$-uniform, which proves that self-dual matroids are uniformly dense. Geiger, Hashimoto, Sturmfels and Vlad proved in \cite{geiger_2022_self-dual} that the base polytope of self-dual matroids contains the point $(1/2,\dots,1/2)$. They called this property ``stable" in reference to stability of point configurations.
\end{example}
%%%%%%%%%%%%%%%%%%%%%%%%%%%%%%%%
%                              %
%%%%%%%%%%%%%%%%%%%%%%%%%%%%%%%%

\subsection{Basic structure and operations}\label{subsection: basic structure and operations}
We start by giving a characterization of two forcing structures in uniformly dense matroids. An element $e\in E$ that is contained in every basis of a matroid is called a \emph{coloop} and an element that is contained in no basis is called a \emph{loop}; the following result shows that we do not lose much by restricting to loopless matroids when studying uniform density.
\begin{theorem}\label{th: no loops or coloops}
If a uniformly dense matroid contains a coloop (respectively loop), then all elements must be coloops (respectively loops).
\end{theorem}
\begin{proof} A coloop $\ell$ satisfies $\mu(\lbrace B:B\ni \ell \rbrace)=\mu(\mathcal{B})$ for any basis measure $\mu$. If $M$ is uniformly dense with $E$-uniform basis measure $\mu$, this implies $\mu(\lbrace B: B\ni e\rbrace)=\mu(\mathcal{B})$ for every $e\in E$, and by Lemma \ref{lemma: foster euler marginal sum}, that $\rank(M)=\vert E\vert$. This characterizes the uniform matroid $U_{n,n}=(E,E)$ in which every element is a coloop. If $M$ is uniformly dense and contains a loop $\hat{\ell}$, then $\infty=\rho(\hat{\ell})\leq \rho(M)$ implies that $\rank(M)=0$, which characterizes the matroid $U_{0,n}=(E,\emptyset)$ in which every element is a loop; since $M$ contains a loop it is not $U_{0,0}$.
\end{proof}

Thus, a uniformly dense matroid of size $n$ is either $U_{n,n}, U_{0,n}$, or it has no loops or coloops.\newline

Next, we show how uniform density behaves with respect to some common operations in matroid theory. Theorems~\ref{th: uniform density and duality} and \ref{th: uniform density and union} were proven by Narayanan and Vartak in \cite{narayanan_1981_molecular}, and Theorem \ref{th: uniform density and direct sum} is new to our knowledge. These results are not used in later sections, but they are included to illustrate that uniformly dense matroids are a well-behaved class of matroids. The first operation is duality: the \emph{dual} of a matroid $M=(E,\mathcal{B})$ is the matroid $$M^\star := (E,\lbrace E\backslash B : B\in \mathcal{B}\rbrace).$$
\begin{theorem}[{\cite[Thm. 4]{narayanan_1981_molecular}}]\label{th: uniform density and duality}
A matroid $M$ is uniformly dense if and only if its dual $M^\star$ is uniformly dense.
\end{theorem}
\begin{proof} We give a new proof based on our new characterization; the proof in \cite{narayanan_1981_molecular} is different but very short as well. Let $M$ be a matroid and $M^\star$ its dual. The bijection between $\mathcal{B}$ and $\mathcal{B}^\star$ via $B^\star=E\backslash B$ induces a bijection between measures $\mu$ on $\mathcal{B}$ and $\mu^\star$ on $\mathcal{B}^\star$. We find
\begin{align*}
\sum_{B\in\mathcal{B}:B\ni e}\mu(B) = \sum_{B\in\mathcal{B}:B\ni e}\mu^\star(E\backslash B) = \sum_{B^\star\in\mathcal{B}^\star:B^\star\not\ni e}\mu^\star(B^\star) = \mu^\star(\mathcal{B}^\star)-\sum_{B^\star\in\mathcal{B}^\star:B^\star\ni e}\mu^\star(B^\star),
\end{align*}
where the last step uses additivity of the measure over the partition $\mathcal{B}^\star=\lbrace B^\star : B^\star\ni e\rbrace\sqcup\lbrace B^\star: B^\star\not\ni e\rbrace.$ Since $\mu^\star(\mathcal{B}^\star)$ is independent of $e$, the measure $\mu$ is $E$-uniform if and only if $\mu^\star$ is $E$-uniform and thus $M$ is uniformly dense if and only $M^\star$ is.
\end{proof}

The \emph{union} of two matroids $M_1=(E_1,\mathcal{B}_1)$ and $M_2=(E_2,\mathcal{B}_2)$ is the matroid $$M_1\lor M_2 := (E_1\cup E_2,\max\lbrace B_1\cup B_2 : B_1\in\mathcal{B}_1,B_2\in\mathcal{B}_2\rbrace),$$ where `$\max$' is taken with respect to the subset partial order. 

\begin{theorem}[{\cite[Cor. 1]{narayanan_1981_molecular}}]\label{th: uniform density and union}
The union of two uniformly dense matroids $M_1$ and $M_2$ on the same ground set $E$ is a uniformly dense matroid $M_1\lor M_2$ on $E$.
\end{theorem}
The following example shows that the converse of Theorem \ref{th: uniform density and union} does not hold: the matroid $M=([3],\lbrace 12,13\rbrace)$ is not uniformly dense but the union $M\lor M=U_{3,3}$ is uniformly dense.

The \emph{intersection} of two matroids $M_1$ and $M_2$ is the matroid 
$$
M_1\land M_2 := (M_1^\star \lor M_2^\star)^\star.
$$
It follows immediately from Theorems \ref{th: uniform density and duality} and \ref{th: uniform density and union} that intersections preserve uniform density.
\begin{corollary}
The intersection of two uniformly dense matroids $M_1$ and $M_2$ on the same ground set is a uniformly dense matroid $M_1\land M_2$ on $E$.
\end{corollary}

The \emph{direct sum} of two matroids $M_1=(E_1,\mathcal{B}_1)$ and $M_2=(E_2,\mathcal{B}_2)$ with disjoint ground sets is the matroid $$M_1\oplus M_2 := (E_1\cup E_2, \lbrace B_1\cup B_2: B_1\in\mathcal{B}_1,B_2\in\mathcal{B}_2\rbrace).$$ In other words, the direct sum is the union of matroids on disjoint ground sets.
\begin{theorem}\label{th: uniform density and direct sum}
The direct sum $M_1\oplus M_2$ of two matroids is uniformly dense if and only if $M_1$ and $M_2$ are uniformly dense with equal density $\rho(M_1)=\rho(M_2)=\rho(M_1\oplus M_2)$.
\end{theorem}
\begin{proof}
From the definition of direct sum, it follows that the base polytope of a direct sum matroid is the cartesian product of the base polytopes of the summands, that is $P(M_1\oplus M_2)=\{ (x,y)\in\mathbb{R}^{\vert E_1\vert+\vert E_2\vert} : x\in P(M_1), y\in P(M_2) \}$. Hence, $P(M_1\oplus M_2)$ contains the constant vector $\rho(M_1\oplus M_2)^{-1}\cdot\mathbbm{1}$ if and only if the polytopes $P(M_1)$ and $P(M_2)$ contain the constant vectors $\rho(M_1\oplus M_2)^{-1}\cdot\mathbbm{1}$ of appropriate lengths. This is possible if and only if $\rho(M_1\oplus M_2)$ equals $\rho(M_1)$ and $\rho(M_2)$ since the base polytopes $P(M_1)$ and $P(M_2)$ lie in the hyperplane of vectors which sum to the rank of the matroid, so if these polytopes contain a constant vector then it must have values equal to the density of the matroid. This completes the proof.
\end{proof}

We now come back to the rank-based characterization of strict uniform density. A matroid is \emph{connected} if it cannot be written as the direct sum of two other matroids. Every matroid $M$ has a partition $E=E_1\sqcup \dots\sqcup E_\ell$ of its ground set such that $M=\bigoplus_{i=1}^\ell M_i$ where each summand $M_i$ is a connected matroid on $E_i$; these are called the \emph{connected components of $M$}.

\begin{theorem}\label{theorem: rank conditions for strict uniform density}
The following are equivalent for a matroid $M$:
\begin{itemize}
    \item[(1)] $P(M)$ contains the point $(\rho^{-1},\dots,\rho^{-1})$ in its relative interior.
    \item[(2)] $\rho(A)\leq \rho(E)$ for all subsets $A\subseteq E$, with strict inequality whenever $A$ is not the union of ground sets of connected components of $M$.
    \item[(3)] There exists a positive measure $\mu$ on $\mathcal{B}$ such that $\mu(\{ B: B\ni e\})$ is equal for all $e\in E$.
\end{itemize}
A matroid that satisfies these conditions is called \emph{strictly uniformly dense}.
\end{theorem}
\begin{proof}
The equivalence between \emph{(1)} and \emph{(3)} follows as in the proof of Theorem \ref{th: main characterization theorem}. To prove equivalence with \emph{(2)}, we use the fact, shown in \cite[Prop. 2.4]{feichtner_2005_matroid}, that $P(M)$ has codimension equal to the number of connected components of $M$. Let $M=\bigoplus_{i=1}^\ell M_i$ be a decomposition into connected components. From the definition of direct sum it follows that $\sum_{e\in E(M_i))} x_e=\rank(E(M_i)$ for all $x\in P(M)$ and for every connected component $M_i$. These are $\ell$ linearly independent equations for points in the base polytope $P(M)$. Since the codimension of $P(M)$ is $\ell$, any further independent face inequality $\sum_{e\in A}x_e\leq \rank(A)$ must determine a proper face of $P(M)$. Thus the constant point $(\rho^{-1},\dots,\rho^{-1})$ lies in the relative interior if and only if these proper face inequalities are strict, which translates to $\rho(A)<\rho(E)$ following the proof of Theorem \ref{th: main characterization theorem}, whenever $A$ is not a union of some of the $E(M_i)$'s. This completes the proof.
\end{proof}
%%%%%%%%%%%%%%%%%%%%%%%%%
%                       %
%%%%%%%%%%%%%%%%%%%%%%%%%

\subsection{Testing uniform density}
As a direct corollary of Theorem~\ref{th: main characterization theorem} and of earlier characterizations due to Naranayan and Vartak \cite{narayanan_1981_molecular} and Catlin et al.\ \cite{catlin_1992_arboricity}, it follows that uniform density of a matroid can be tested efficiently, relative to an independence oracle that answers the query ``is $\rank(A)=\vert A\vert$ or not?". We include the result here, as we have not found it reported elsewhere in the literature.
\begin{corollary}\label{corollary: testing uniform density}
There is a polynomial-time algorithm that tests if a matroid $M$ is uniformly dense and, if it is not, outputs a density inequality that is violated.
\end{corollary}
\begin{proof}
Cunningham \cite{cunningham_1984_testing} developed a polynomial-time algorithm that tests if a vector $p$ lies in a given matroid basis polytope $P(M)$, and otherwise outputs a violating hyperplane inequality. This can be used to check if the vector $(\rho^{-1},\dots,\rho^{-1})$ lies inside $P(M)$ or not and thus if $M$ is uniformly dense. A second approach is to construct the matroid $M'$ which contains $\rank(M)$ copies of every element $e\in E(M)$ and then use Edmonds' matroid partitioning algorithm \cite{edmonds_1965_minimum} to test if $E(M')$ can be covered by $\rank(M')=\rank(M)$ bases of $M'$.
\end{proof}
%%%%%%%%%%%%%%%%%%%%%%%%%
%                       %
%%%%%%%%%%%%%%%%%%%%%%%%%

\section{Graphic matroids}\label{section: graphic matroids}
We now consider matroids which are associated with graphs and study some implications of uniform density for the structure of these graphs. The main results are Theorem \ref{th: 2-cycle uniformly dense graphs}, which characterizes bicyclic connected uniformly dense graphs, Theorem \ref{th: toughness of uniformly dense graphs}, which gives a strong connectivity result for uniformly dense graphs and Theorem \ref{th: spectral characterization}, which gives a spectral characterization of uniformly dense graphs.

\subsection{Definition and examples}
A \emph{graph} $G=(V,E)$ consists of a finite set $V$ and a set $E\subseteq \lbrace \lbrace u,v\rbrace : u,v\in V\rbrace$ of $2$-element subsets of $V$. Elements of $V$ are called \emph{vertices} and elements of $E$ are called \emph{edges}. We will only consider \emph{simple graphs}, where every pair of vertices can appear at most once as an edge, and where edges from a vertex to itself are not allowed. For a subset $A\subseteq E$ of edges, we write $G[A]$ for the subgraph of $G$ obtained by removing all edges not in $A$ and any isolated vertices.

A \emph{spanning forest} in $G$ is a maximal subset of edges $T\subseteq E$ which contains no cycle; if $G$ is connected then so is $T$ and it is called a spanning tree. The collection of spanning forests of a graph determines a matroid
$$
M(G) = (E,\lbrace T\subseteq E\,:\, T \text{~is a spanning forest in $G$}\rbrace).
$$
$M(G)$ is called the \emph{cycle matroid} of $G$, and any matroid which can be realized as the cycle matroid of a (multi-)graph is called a \emph{graphic matroid}. In general, a single matroid can be realized by different graphs (as we shall see for instance in Example \ref{ex: tree graph}). 

The \emph{rank} of a graph is $\rank(G)=\rank(M(G))$, and the rank of a subset of edges is defined as the rank of the corresponding subset in $M(G)$. Graph theoretically, the rank of a graph is related to the number of components and the number of cycles in the graph. The \emph{number of components} $c(A)$ of a subset $A\subseteq E$ is defined as
$$
c(A) := \vert V\vert - \rank(A),
$$
and we write $c(G)$ for the number of components of a graph\footnote{The number of components $c(G)$ of a graph does not correspond to the number of components of the matroid $M(G)$. Instead, the number of components of $M(G)$ equals the number of \textit{biconnected} components of $G$ (see \cite[Prop. 4.1.7]{oxley_2011_matroid}); a biconnected component is a maximal subgraph without cut vertices.}. This terminology reflects that $c(A)$ is the number of connected components in $G[A]$; see \cite[1.3.8]{oxley_2011_matroid}. We note that $c(A)$ requires the information of the graph and not just the matroid; indeed, in general the number of vertices of a graph cannot be uniquely recovered from the matroid $M(G)$. A \emph{connected graph} is a graph $G$ with $c(G)=1$. The \emph{number of independent cycles} in a subset $A\subseteq E$ is defined as 
$$
\beta(A):=\vert A\vert - \rank(A),
$$ 
and equals the number of linearly independent cycles in the set $A$, i.e., the dimension of the cycle space of $G[A]$; we write $\beta=\beta(G)$ for the number of independent cycles in a graph. The density $\rho(G):=\rho(M(G))$ of a graph can thus be expressed as
\begin{equation}\label{eq: density of a graph}
\rho(G) = \frac{\vert E\vert}{\rank(G)} = \frac{\vert E\vert}{\vert V\vert - c(G)} = \frac{\vert E\vert}{\vert E\vert - \beta(G)}.    
\end{equation}

A \emph{uniformly dense graph} is a graph $G$ whose cycle matroid $M(G)$ is uniformly dense. Since the cycle matroid of a graph is the direct sum of the cycle matroids of the connected components of the graph, it follows from Theorem \ref{th: uniform density and direct sum} that a graph is uniformly dense if and only if each connected component is uniformly dense with equal density. In the rest of this section, we study properties of uniformly dense graphs.

\begin{remark}\label{rmk:sparse}
    The name ``density'' for matroids corresponds to what is classically known as density for simple graphs. A graph $G$ is called \emph{dense} if the number of edges is close to the maximal possible number of edges, and it is called \emph{sparse} otherwise. Hence, the complete graph is the most dense simple graph, and the further $G$ is from the complete graph, the sparser it is. Also, the denser $G$ is, the larger $\vert E\vert$ is compared to $\vert V\vert$, and therefore the larger $\rho(G)$ is, following \eqref{eq: density of a graph}. This gives the following interpretation of uniform density in graphs: a graph $G$ is uniformly dense if, for all nonempty $A\subseteq E$, the graph $G[A]$ is at most as dense as $G$. 
\end{remark}

\begin{example}\label{example: spanning trees of the 4 cycle with diagonal}
Figure \ref{fig: example of graphic matroid} shows a graph $G$ with $\vert V\vert=4$ vertices and $\vert E\vert=5$ edges. The graph is connected, that is, $c(G)=1$, and thus it has $\rank(G)=\vert V\vert-1 = 3$. On the right, all spanning forests are shown --- these are the sets of $3=\rank(G)$ edges that cover all vertices. The cycle matroid $M=M(G)$ corresponding to $G$ has ground set $E(M)=[5]$, given by the labeling of the edges of $G$, and $8$ bases $\mathcal{B}(M)=\lbrace 123,234,134,124,145,245,135,235\rbrace$ given by the spanning forests of $G$. The 3-element non-bases of $M$ are $\lbrace 125, 345\rbrace$ and correspond to the $3$-cycles in the graph.
\begin{figure}[h!]
    \centering
    \includegraphics[width=0.9\textwidth]{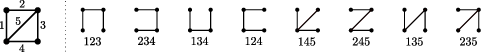}
    \caption{A graph $G$ and its spanning forests. The cycle matroid of $G$ has ground set $E=[5]$ and bases given by the three-element subsets indicated under the spanning forests.}
    \label{fig: example of graphic matroid}
\end{figure}
\end{example}

\begin{example}[Tadpole graph]\label{ex: tadpole graph} 
The tadpole graph (also called paw graph) shown in Figure \ref{fig: examples of uniformly dense graphs} looks like a tadpole and has the tadpole matroid $T_4$ as cycle matroid. As shown in Example \ref{ex: tadpole matroid}, the tadpole graph is not uniformly dense.
\end{example}

\begin{example}[Forest graph]\label{ex: tree graph}
The cycle matroid of a forest graph with $n$ edges is the uniform matroid $U_{n,n}$. This illustrates how one matroid ($U_{n,n}$) can correspond to multiple graphs (all forest graphs on $n$ edges). Since the uniform matroid is uniformly dense, forest graphs are uniformly dense. A \emph{cut edge} in a graph is an edge $e$ such that $c(E\backslash e)=c(E)+1$ and corresponds to a coloop in the cycle matroid. Following Theorem \ref{th: no loops or coloops}, a uniformly dense graph is either a forest graph or it contains no cut edges. A loop in a graph is an edge of the form $\lbrace v,v\rbrace$ and corresponds to a loop in the cycle matroid. Loops are excluded in our definition of graphs, but this is without much loss of generality when studying uniformly dense graphs due to Theorem \ref{th: no loops or coloops}.
\end{example}
\begin{example}[Edge-transitive graphs]\label{ex: edge-transitive graphs}
An \emph{automorphism} of a graph $G=(V,E)$ is a permutation $\pi$ of $V$ such that $$\lbrace \pi(u),\pi(v)\rbrace \in E \iff \lbrace u,v\rbrace\in E.$$ The automorphisms of a graph form a group $\operatorname{Aut}(G)$ and a graph is \emph{edge transitive} if $\operatorname{Aut}(G)$ acts transitively on $E$ via $\pi\lbrace u,v\rbrace=\lbrace \pi(u),\pi(v)\rbrace$. In other words, $G$ is edge transitive if for all $e,f\in E$ there exists an automorphism $\pi\in\operatorname{Aut}(G)$ such that $\pi(e)=f$. Narayanan and Vartak showed in \cite{narayanan_1981_molecular} that edge-transitive graphs are uniformly dense. A proof follows from the observation that the uniform measure on the bases of an edge-transitive graph is $E$-uniform. Figure \ref{fig: examples of uniformly dense graphs} shows some common examples of edge-transitive graphs: the cycle graph, complete graph and hypercube graph.
\end{example}
\begin{figure}[h!]
    \centering
    \includegraphics[width=0.8\textwidth]{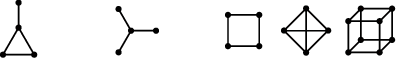}
    \caption{From left to right: (Example \ref{ex: tadpole graph}) The tadpole graph has density $\rho=4/3$ and is not uniformly dense. (Example \ref{ex: tree graph}) The tree graph has density $\rho=1$ and is uniformly dense. (Example \ref{ex: edge-transitive graphs}) The $4$-cycle graph has density $\rho=4/3$ and is uniformly dense, the complete graph on $4$ vertices has density $\rho=2$ and is uniformly dense and the hypercube graph of dimension $3$ has density $\rho=12/7$ and is uniformly dense. The latter four graphs are edge-transitive.}
    \label{fig: examples of uniformly dense graphs}
\end{figure}
%%%%%%%%%%%%%%%%%%%%%%%%%
%                       %
%%%%%%%%%%%%%%%%%%%%%%%%%

\subsection{Structure}\label{subsection: structure of ud graphs}
To illustrate the implications of uniform density in graphs, we consider a number of structural properties. The \emph{degree} $\deg(v):=\vert\lbrace e\in E : e\ni v\rbrace\vert$ of a vertex $v\in V$ is equal to the number of edges that contain $v$; recall that loops are excluded, which otherwise should be counted twice in computing the degree. We write $\delta(G)$ for the smallest degree of any vertex, and $\Delta(G)$ for the largest degree. A graph is \emph{regular} if all vertices have the same degree.
\begin{proposition}\label{prop: minimal degree}
If $G$ is a uniformly dense graph, then $\delta(G)\geq \rho(G)$.
\end{proposition}
\begin{proof}
Let $v\in V$ be a minimal degree vertex, and let $A_v:=\lbrace e\in E : e\ni v\rbrace$ be the set of edges that contain $v$. Note that $\vert A_v\vert\geq\delta(G)$. Since removing $A_v$ creates at least one new component, i.e., the separated vertex $v$, we find that 
$$
\rank(E\setminus A_v) = \vert V\vert -c(E\setminus A_v)\leq \vert V\vert - c(E) - 1 \Rightarrow \rank(E)-\rank(E\setminus A_v)\geq 1.
$$
From uniform density of $G$, it then follows that
$$
\rho(E\setminus A_v) = \frac{\vert E\vert - \delta(G)}{\rank(E\setminus A_v)} \leq \frac{\vert E\vert}{\rank(E)} \iff \delta(G)\geq \rho(G)[\rank(E)-\rank(E\setminus A_v)]\Rightarrow \delta(G)\geq \rho(G)
$$
which completes the proof.
\end{proof}

A \emph{clique} in a graph $G$ is a set of vertices $U\subseteq V$ such that $\lbrace u,v\rbrace\in E$ for every pair of distinct vertices $u,v\in U$. The clique number $\operatorname{cl}(G)$ is the largest number of vertices in a clique in $G$.
\begin{proposition}
If $G$ is a uniformly dense graph, then $\operatorname{cl}(G)\leq 2\rho(G)$.
\end{proposition}
\begin{proof}
Let $G$ be a uniformly dense graph and $A$ the edges in a maximal clique of size $k=\operatorname{cl}(G)$. Then
$$
\rho(E) \geq \rho(A)=\frac{\vert A\vert}{\vert V\vert - c(A)} = \frac{\frac{k(k-1)}{2}}{k-1}=\frac{k}{2}.
$$
\end{proof}

The \emph{girth} $\operatorname{gir}(G)$ of a graph $G$ is the size of the smallest cycle in $G$.
\begin{proposition}\label{proposition: girth}
If $G$ is a uniformly dense graph with a cycle, then $\operatorname{gir}(G)\geq \rho(G)/(\rho(G)-1)$.
\end{proposition}
\begin{proof}
Let $G$ be a uniformly dense graph with a cycle, and let $A$ be the set of edges in a smallest cycle of size $k=\operatorname{gir}(G)$. Then,
$$
\rho(E)\geq\rho(A)=\frac{\vert A\vert}{\vert V\vert-c(A)}=\frac{k}{k-1}.
$$
This simplifies to $k\geq \rho(G)/(\rho(G)-1)$ and completes the proof.
\end{proof}

The definition of girth extends to matroids. A circuit is a subset $C\subseteq E$ such that $\rank(C)=\vert C\vert-1$ and $\rank(S)=\vert S\vert$ for every proper subset $S\subset C$; the girth $\operatorname{gir}(M)$ of a matroid is the size of the smallest circuit. The proof of Proposition \ref{proposition: girth} directly extends to general matroids:
\begin{proposition}
If $M$ is uniformly dense with a circuit, then $\operatorname{gir}(M)\geq \rho(G)/(\rho(G)-1)$.
\end{proposition}

When graphs have a small number of cycles, there is a strong interplay between the lengths of these cycles and the density of the graph. As an application, we get the following characterization for connected uniformly dense graphs with two cycles. 
\begin{theorem}\label{th: 2-cycle uniformly dense graphs}
Let $G$ be a connected graph with $\beta(G)=2$. Then $G$ is uniformly dense if and only if $G$ can be constructed from three paths identified at the ends, with lengths $L_1,L_2$ and $L_3$ that satisfy $(L_3-L_2)\leq L_1\leq L_2\leq L_3$.
\end{theorem}
In the case where one of the paths has length zero, the graph in Theorem \ref{th: 2-cycle uniformly dense graphs} consists of two cycles of equal length identified in a single vertex (see Example \ref{example: bicyclic graphs example} further below). In the proof below, the forward direction was suggested by Harry Richman and simplifies an earlier version.

\begin{proof} Note that since $G$ is connected, $\beta(G)=2$ implies that $\vert E\vert=\vert V\vert+1$. 

(Proof of forward direction.) Let $G$ be a graph constructed from three paths identified at the ends. Let $P_1,P_2,P_3$ denote the edge sets of the paths, and let $L_1,L_2,L_3$ be their respective lengths. We show that the dual matroid $M(G)^\star$ is uniformly dense if and only if the inequalities of the theorem hold. Since $G$ is a planar graph, the dual matroid $M(G)^\star$ is again a graphic matroid, with corresponding graph given by the planar dual $G^\star$ of $G$ (see \cite[Thm. 2.3.4]{oxley_2011_matroid}); Figure \ref{fig: example of graph dual} shows an example. The graph $G^\star$ is a multigraph on $3$ vertices, say $\{ a,b,c\}$ with $L_1$ edges $E_{ab}$ between $a$ and $b$, $L_2$ edges $E_{bc}$ between $b$ and $c$ and $L_3$ edges $E_{ac}$ between $a$ and $c$; see the figure below for illustration. $M(G)^\star$ has rank two and its bases are given by pairs of edges that do not lie between the same vertices. 
\begin{figure}[h!]
    \centering    \includegraphics[width=0.6\textwidth]{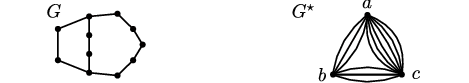}
    \caption{A planar graph $G$ consisting of paths of length $L_1=3,L_2=3,L_3=5$. The three vertices of the dual graph $G^\star$ correspond to the two bounded and one unbounded face of $G$.}
    \label{fig: example of graph dual}
\end{figure}

We now show that the path-length inequalities imply the density inequalities for all non-empty subsets $A\subset E=E_{ab}\sqcup E_{bc}\sqcup E_{ac}$. The subsets $A$ of rank $2$ are obtained by taking a basis and adding any further elements; in this case, $\rho(A)=\vert A\vert/2\leq \vert E\vert/2 = \rho(E)$. All remaining nonempty subsets $A$ have rank $1$ and are obtained by taking a non-empty subset of $E_{ab},E_{bc}$ or $E_{bc}$. Since the density satisfies
$\rho(A) = \vert A\vert$, the extreme cases for which the density inequalities need to be checked correspond to when $A$ is equal to one of the three sets of parallel edges, in these cases $\rho(E_{ab})=L_1$, $\rho(E_{bc})=L_2$ and $\rho(E_{ac})=L_3$. We check the density inequalities:
\begin{align*}
\rho(E_{ab})&\leq \rho(E) \iff 2L_1\leq L_1+L_2+L_3\iff  L_1\leq L_2+L_3\Longleftarrow L_1\leq L_2\\
\rho(E_{bc})&\leq \rho(E)\iff 2L_2\leq L_1+L_2+L_3\iff L_2\leq L_1+L_3\Longleftarrow L_2\leq L_3\\
\rho(E_{ac})&\leq \rho(E)\iff 2L_{3}\leq L_1+L_2+L_3\iff L_3\leq L_1+L_2\Longleftarrow L_1\geq (L_3-L_2).
\end{align*}
Thus, uniform density of $M(G)^\star$ follows from the assumptions on the path lengths. This completes the forward direction.
\\
(Proof of converse direction.) Let $G$ be a connected uniformly dense graph with $\vert E\vert = \vert V\vert+1$ edges. We will show that $G$ can be constructed from three paths identified at the ends. The conditions on the lengths then follow from the proof of the forward direction above.

By Kuratowski's Theorem \cite[Thm. 2.3.8]{oxley_2011_matroid} and $\beta(G)=2$, the graph $G$ is planar. Its planar dual $G^\star$ has $\rank(G^\star)=\vert E\vert-\rank(G)=\beta(G)=2$. Since the operation of collapsing parallel multi-edges to single edges does not change the rank of a graph, the multi-graph $G^\star$ with all parallel multi-edges collapsed is a simple graph of rank $2$; this is either a $2$-path or a $3$-cycle. This means that $G$ is the planar dual of a $2$-path or a $3$-cycle graph, with some edges duplicated into parallel multi-edges, and thus proves that $G$ has the required structure of three paths identified at their ends (each path dual to one of the multi-edges). This completes the proof.
\end{proof}

\begin{example}\label{example: bicyclic graphs example}
Let $G,H,K$ be the graphs in Figure \ref{fig: example of bicyclic graphs}. These graphs are constructed by identifying the ends of three paths of lengths $(0,4,4)$ for $G$, lengths $(3,3,6)$ for $H$ and lengths $(2,3,6)$ for $K$. Applying Theorem \ref{th: 2-cycle uniformly dense graphs}, we find that $G$ is uniformly dense (since $4-4\leq 0\leq 4\leq 4$) and $H$ is uniformly dense (since $6-3\leq 3\leq 3\leq 6$) but $K$ is not uniformly dense (since $6-3\not\leq 2\leq 3\leq 6$).
\begin{figure}[h!]
    \centering
    \includegraphics[width=0.6\textwidth]{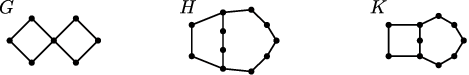}
    \caption{The three graphs $G,H$ and $K$ from Example \ref{example: bicyclic graphs example}. $G$ and $H$ are uniformly dense and $K$ is not uniformly dense.}
    \label{fig: example of bicyclic graphs}
\end{figure}
\end{example}
%%%%%%%%%%%%%%%%%%%%%%%%
%                      %
%%%%%%%%%%%%%%%%%%%%%%%%

\subsection{Connectivity}\label{subsection: connectivity for ud graphs}
Uniform density has strong implications for the connectivity of a graph. Recall that $c(A)$ is the number of connected components of the graph $G[A]$ and thus that $c(E\backslash A)-c(E)$ is the number of new components created when removing a set of edges $A$ from the graph. The uniform density condition of a graph can be written in terms of the number of connected components as follows:
\begin{proposition}\label{prop: uniform density and connected components}
A graph $G$ is uniformly dense if and only if
\begin{equation}\label{eq: uniform density for graphs and connected components}
c(E\backslash A) -c(E)\leq \frac{\vert A\vert}{\rho(G)}\text{~for all $A\subset E$.}
\end{equation}
\end{proposition}
\begin{proof}
Let $G$ be a graph and $A\subset E$ a proper subset of edges. Then,
\begin{align*}
\rho(E\backslash A)\leq \rho(E) 
&\iff \frac{\vert E\backslash A\vert}{\rank(E\backslash A)} \leq \frac{\vert E\vert}{\rank(E)} 
\\
& \iff \frac{\vert E\vert - \vert A\vert}{\vert V\vert - c(E\backslash A)}\leq \frac{\vert E\vert}{\vert V\vert - c(E)}
\\
& \iff c(E\backslash A)-c(E) \leq \vert A\vert \frac{\vert V\vert-c(E)}{\vert E\vert}
\\
&\iff c(E\backslash A)-c(E) \leq \frac{\vert A\vert}{\rho(G)}.
\end{align*}
Since $E\backslash A$ can be any nonempty set of edges, this completes the proof.
\end{proof}

In other words, removing $k$ edges from a uniformly dense graph increases the number of components by at most $k/\rho(G)$. This notion of connectivity is also called ``strength'' and its reciprocal ``vulnerability'' and was studied by Gusfield in \cite{gusfield_1983_connectivity} for graphs, by Cunningham in \cite{cunningham_1985_optimal} for matroids. Strength was related to uniform density by Catlin et al.\ in \cite{catlin_1992_arboricity}.

A closely related but vertex-based notion of connectivity is toughness. A graph is \emph{$t$-tough} if removing any $k$ vertices increases the number of components by at most $k/t$. The complete graph is defined to be $t$-tough for every $t$. See \cite{bauer_2006_toughness} for a survey of toughness and its relation to other graph properties. Any graph is $(1/\Delta(G))$-tough, but for uniformly dense graphs this can be improved by a factor of $\rho(G)$, and even more for connected regular graphs.
\toughness
\begin{proof}
Let $G$ be a uniformly dense graph. Let $U\subseteq V$ be a subset of vertices whose removal increases the number of components by $n_U>0$, and let $A_U:=\lbrace \lbrace u,v\rbrace\in E: u\in U,v\not\in U\rbrace$ be the cut-set of $U$. Removing the vertices $U$ from $G$ is the same as first removing the edges $A_U$ and then the vertices $U$ from $G$. As a result, removing $A_U$ results in at least as many new components as removing $U$ and thus $n_U\leq c(E\backslash A_U)-c(E)$. By uniform density of $G$ and invoking Proposition \ref{prop: uniform density and connected components}, we then find
$$
n_U\leq c(E\backslash A_U) - c(E) \leq \frac{\vert A_U\vert}{\rho(G)} \leq \frac{\Delta(G)}{\rho(G)}\vert U\vert.
$$
This proves that uniformly dense graphs are $(\rho(G)/\Delta(G))$-tough.

Let now $G$ be a connected, $d$-regular, uniformly dense graph. Let $U\subseteq V$ be a subset of vertices whose removal from $G$ increases the number of components, and let $T$ be a spanning tree of $G$ (since $G$ is connected). We write $G'$ for the graph $G$ with vertices $U$ removed and $F':=\lbrace e\in T : e\subseteq V\backslash U\rbrace$. We note that $F'$ is not necessarily a spanning tree or forest of $G'$, and in general we have $c(F')\geq c(G')$. See Figure \ref{fig: graph trees and subgraphs} for an example of the construction.

\begin{figure}[h!]
    \centering
    \includegraphics[width=0.5\textwidth]{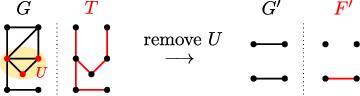}
    \caption{A graph $G$ and spanning tree $T\subset E(G)$. The graph $G'$ and $F'\subset E(G')$ are obtained by removing the vertices in $U$ from $G$. Note that in this example $F'$ is not a spanning forest of $G'$, and $c(F')> c(G')$.}
    \label{fig: graph trees and subgraphs}
\end{figure}

Let $\deg_T(v):=\vert\lbrace e\in T : e\ni v\rbrace\vert$ denote the \emph{degree} of a vertex $v$ in $T$, and let $$\epsilon_U(T):=\vert \lbrace e\in T : e\subseteq U \rbrace\vert$$ be the number of edges in $T$ between pairs of vertices in $U$.
For the edge set $F'$, we know that $\vert F'\vert = \vert V\setminus U\vert - c(F')$ since the set contains no cycles, and therefore we can write
\begin{align}
0 &= \vert F'\vert - \vert V\setminus U\vert + c(F')\nonumber
\\
&= \left[\vert T\vert - \left(\sum_{v\in U}\deg_T(v) - \epsilon_U(T)\right)\right] - \Big[\vert V\vert - \vert U\vert\Big] + c(F')\nonumber
\\
&= (\vert T\vert - \vert V\vert +1)-1 - \sum_{v\in U}\deg_T(v) + \epsilon_U(T) + \vert U\vert + c(F').\nonumber
\end{align}
Since $G$ is connected, we have $\vert T\vert = \rank(G)=\vert V\vert-1$ and the equality above further simplifies to
\begin{equation}\label{eq: sum over degrees in tree for toughness}
\sum_{v\in U} \deg_T(v) = \epsilon_U(T) + \vert U\vert + c(F') - 1. 
\end{equation}
By uniform density, there exists a normalized $E$-uniform measure $\mu$ on the spanning trees $\mathcal{B}$ of $G$. By Lemma \ref{lemma: foster euler marginal sum} this measure satisfies $\mu(\lbrace T\in\mathcal{B} : T\ni e\rbrace)=\rho(G)^{-1}$ for every edge $e$. For any vertex $v\in V$ and summing over spanning trees, we find 
\begin{align*}
\sum_{T\in\mathcal{B}}\mu(T)(2-\deg_T(v)) &= \sum_{T\in\mathcal{B}}\mu(T)\left(2-\sum_{e\in T:e\ni v}1\right) 
\\
&= 2-\sum_{T\in\mathcal{B}}\sum_{e\in T:e\ni v}\mu(T)\qquad\qquad(\mu\text{~is normalized})
\\
&= 2-\sum_{e\in E:e\ni v}\sum_{T\in\mathcal{B}:T\ni e}\mu(T)\qquad(\text{switch summation order})
\\
&= 2-\frac{d}{\rho(G)}\qquad\qquad\qquad\qquad\text{~($G$ is $d$-regular)}
\\
&= 2 - \frac{d(\vert V\vert-1)}{\vert E\vert} = 2 - \frac{2(\vert V\vert-1)}{\vert V\vert}>0.
\end{align*}
Summing this inequality over the vertices in $U$, we obtain
\begin{align*}
0&<\sum_{v\in U}\sum_{T\in\mathcal{B}}\mu(T)(2-\deg_T(v)) 
\\
&= \sum_{T\in\mathcal{B}}\mu(T)\left(2\vert U\vert - \sum_{v\in U}\deg_T(v)\right)
\\
&= \sum_{T\in\mathcal{B}}\mu(T)\left(2\vert U\vert - \vert U\vert -\epsilon_U(T)-c(F')+1\right) \quad\text{~(Eq.\ \eqref{eq: sum over degrees in tree for toughness})}
\\
&\leq \sum_{T\in\mathcal{B}}\mu(T)\left(\vert U\vert - c(G') + 1 \right)\quad\text{~($\epsilon_U(T)\geq 0$ and $c(F')\geq c(G'))$}
\\
&= \vert U\vert - c(G') + 1.
\end{align*} 
Since $c(G')$ and $\vert U\vert$ are integers, the strict inequality implies $c(G')\leq \vert U\vert$. We recall that $G'$ is the graph $G$ with vertices $U$ removed. This completes the proof that $G$ is $1$-tough.
\end{proof}

Our proof of the $1$-tough property in Theorem \ref{th: toughness of uniformly dense graphs} is heavily based on the proof of Theorem 3.4.18 in Fiedler's book \cite{fiedler_2011_matrices}, which works with a geometric condition on graphs to guarantee the existence of a measure $\mu$ such that $\sum_{T}\mu(T)(2-\deg_T(v))$ is positive for each vertex $v$. This condition is satisfied for connected, regular, uniformly dense graphs.

To conclude, we highlight one implication of the connectivity Theorem \ref{th: toughness of uniformly dense graphs} for uniformly dense graphs. A \emph{perfect matching} in a graph is a subset $A\subseteq E$ of edges such that $A$ covers every vertex exactly once.
\begin{corollary}
Let $G$ be a connected, regular, uniformly dense graph. Then, either $G$ has a perfect matching, or $G\backslash v$ has a perfect matching, for every vertex $v\in V(G)$.
\end{corollary}
\begin{proof}
An even $1$-tough graph has a perfect matching (see Theorem 91 in \cite{bauer_2006_toughness}) and an odd $1$-tough graph has a perfect matching in every subgraph with one vertex removed (see Theorem 92 in \cite{bauer_2006_toughness}). The claim thus follows immediately from Theorem \ref{th: toughness of uniformly dense graphs}.
\end{proof}
%%%%%%%%%%%%%%%%%%%%%%
%                    %
%%%%%%%%%%%%%%%%%%%%%%

\subsection{Spectrum}\label{subsection: spectral results for ud graphs}
In this section, we study spectral properties of uniformly dense graphs; that is, we define linear operators associated with graphs and study how uniform density is reflected in the eigenvalues of these operators. For more background on spectral graph theory, we refer to \cite{MHJ, chung, brouwer_2011_spectra, mohar_laplacian_1991, godsil_algebraic_2001}.\newline

We again work with a simple graph $G=(V,E)$ and note that removing or adding isolated vertices to a graph does not change the density or uniform density. We recall that $G[A]$ denotes the induced subgraph on $A\subseteq E$, in which the edges not in $A$ and any isolated vertices have been removed; we write $V_A:=V(G[A])$ for the vertex set of this induced subgraph. \newline

Let $C(V)$ denote the vector space of functions $f:V\rightarrow\mathbb{R}$, and let $C(E)$ be the vector space of functions $\gamma:E\rightarrow\mathbb{R}$. The \emph{normalized Laplacian} $L=L(G)$ associated with $G$ is the linear map $L:C(V)\rightarrow C(V)$ defined by
$$
Lf(v) : = f(v) - \frac{1}{\deg(v)}\sum_{u\in N_G(v)}f(u),
$$
for $f\in C(V)$ and $v\in V$, where $N_G(v)=\{u\in V\,:\{u,v\}\in E\}$ is the neighbour set of $v$. In matrix notation, this operator is usually written as $L=I-AD^{-1}$ where $I$ is the identity matrix, $A$ the adjacency matrix and $D=\operatorname{diag}(\deg(v_1),\dots,\deg(v_n))$ the diagonal degree matrix.
\newline

Now, for each edge, we fix an arbitrary \emph{orientation}, that is, we let one of its endpoints be its \emph{head}, and we let the other endpoint be its \emph{tail}. Fixing an orientation is needed in order to give the following definition, but changing the orientation of any edge does not affect the following operator.\newline

The \emph{edge Laplacian} $L^1=L^1(G)$ associated with $G$ is the linear map $L^1:C(E)\rightarrow C(E)$ defined by
\begin{equation*}
    L^1\gamma(e):=\frac{\sum_{e_1:v\rm{~tail}} \gamma(e_1)-\sum_{e_2:v\rm{~head}}\gamma(e_2)}{\deg(v)}-\frac{\sum_{e'_1:w\rm{~tail}}\gamma(e'_1)-\sum_{e'_2:w\rm{~head}}\gamma(e'_2)}{\deg(w)},
\end{equation*}
 for $\gamma\in C(E)$ and $e\in E$ that has tail $v$ and head $w$.\newline

\begin{equation*}
    \RQ(f):=\frac{\sum_{\lbrace v,w\rbrace\in E}\biggl(f(v)-f(w)\biggr)^2}{\sum_{v\in V}\deg(v)\cdot f(v)^2},
\end{equation*}for $f\in C(V)$.\newline

As shown, for instance, in \cite{MHJ},\cite{chung}, $L$ has $\vert V\vert$ real, nonnegative eigenvalues whose algebraic and geometric multiplicity coincide, and similarly $L^1$ has $\vert E\vert$ real, nonnegative eigenvalues. The nonzero eigenvalues of the two operators coincide, and they are known to encode several geometric properties of the graph associated with them. Moreover, according to the Courant--Fischer--Weyl min-max Principle, the eigenvalues of $L$ are given by minimax values of the Rayleigh quotients \cite{chung}
\begin{equation*}
    \RQ(f):=\frac{\sum_{\lbrace v,w\rbrace\in E}\biggl(f(v)-f(w)\biggr)^2}{\sum_{v\in V}\deg(v)\cdot f(v)^2},
\end{equation*}for $f\in C(V)$.\newline

A first spectral result on uniformly dense graphs follows from the fact that the multiplicity of zero as an eigenvalue of $L(G)$ is equal to the number of connected components of $G$ and the multiplicity of zero as an eigenvalue of $L^1(G)$ is equal to the number of independent cycles in $G$. Let $n_0(\cdot)$ denote the \emph{multiplicity of zero as an eigenvalue} of a matrix, then we have \cite{MHJ}
\begin{equation}\label{eq: nullity connected components and cycles}
n_0(L(G))=c(G) \quad\text{~and~}\quad n_0(L^1(G))=\beta(G).    
\end{equation}
The uniform density condition of a graph can be written in terms of the Laplacian spectra as follows:
\begin{theorem}\label{th: spectral characterization}
The following are equivalent for a graph $G$:
\begin{enumerate}
    \item[(1)] $G$ is uniformly dense.

    \item[(2)] $\rank(L(G))-\rank(L(G[A])) \leq \frac{\vert E\vert-\vert A\vert}{\rho(G)}$
    for all $\emptyset \neq A\subseteq E$.

    \item[(3)] $\frac{n_0(L^1(G[A]))}{\vert A\vert} \leq \frac{n_0(L^1(G))}{\vert E\vert} \text{~for all $\emptyset \neq A\subseteq E$.}$
\end{enumerate}
\end{theorem}
\begin{proof}
$(1)\Leftrightarrow (2).$ Note that $\rank(L(G))=\vert V\vert-n_0(L(G))$. Following \eqref{eq: nullity connected components and cycles}, the difference of the Laplacian ranks can be written as
\begin{align*}
\rank(L(G))-\rank(L(G[A])) &= (\vert V\vert - c(G)) - (\vert V_A - c(G[A])) 
\\
&= \rank(E) - \rank(A)
\\
&= c(A) - c(E).
\end{align*}
The rank in the first line is the linear-algebraic rank of the Laplacian matrices, while the rank in the second line is the rank in the cycle matroid $M(G)$. By Proposition \ref{prop: uniform density and connected components} we know that $c(A)-c(E)\leq (\vert E\vert - \vert A \vert)/\rho(G)$ if and only if $G$ is uniformly dense; this proves the equivalence.

$(1)\Leftrightarrow (3).$  We recall that $\beta(G)=\vert E\vert-\rank(G)$. Invoking expression \eqref{eq: nullity connected components and cycles}, we then find
\begin{align*}
\frac{n_0(L^1(G[A]))}{\vert A\vert} \leq \frac{n_0(L^1(G))}{\vert E\vert} &\iff \frac{\vert A\vert - \rank(A)}{\vert A\vert} \leq \frac{\vert E\vert-\rank(E)}{\vert E\vert}\\
& \iff \frac{\rank(A)}{\vert A\vert} \geq \frac{\rank(E)}{\vert E\vert}
\\
&\iff \rho(A)\leq \rho(E),
\end{align*}
for all $\emptyset\neq A\subseteq E(G)$. By Theorem \ref{th: main characterization theorem} the latter inequality is equivalent to uniform density.
\end{proof}

Now, we let $\lambda_{\max}(G)$ denote the largest eigenvalue of $L(G)$, and similarly, given $A\subseteq E$, we let $\lambda_{\max}(G[A])$ denote the largest eigenvalue of $L(G[A])$. We also define the \emph{boundary} of $A$ as 
\begin{equation*}
    \delta(A):=\{e\in E\setminus A\,:\, \text{at least one endpoint of }e\text{ is in }V_A\,\}.
\end{equation*} Given a set $S$ of vertices, we let
\begin{equation*}
    \vol(S):=\sum_{v\in S}\deg v
\end{equation*}be the \emph{volume} of $S$. Moreover, given a vertex $v$, we let $\deg_A(v)$ denote the degree of $v$ in the graph $G[A]$, and given $S\subseteq V_A$, we let
\begin{equation*}
    \vol_A(S):=\sum_{v\in S}\deg_A(v).
\end{equation*}

\begin{theorem}\label{thm:lambda-max}
Let $G$ be a graph. Then
\begin{equation}\label{eq:lambdamax}
\lambda_{\max}(G)\geq 2\cdot \rho(G)^{-1}
\end{equation}
and
\begin{equation}\label{eq:A1}
\lambda_{\max}(G)\geq 
2\cdot \rho(A)^{-1} \frac{\vert A\vert}{\vert A\vert + \vert \delta(A)\vert} \quad\text{~for all $\emptyset\neq A\subseteq E$.}   
\end{equation}
Moreover, if $G$ is uniformly dense, then
\begin{equation}\label{eq:lambdamax-stable}
\lambda_{\max}(G[A])\geq 2\cdot \rho(G)^{-1} \quad \text{for all }\emptyset\neq A\subseteq E.
\end{equation}
\end{theorem}
\begin{proof}
Let $G$ be a graph with $n$ vertices and $m$ edges and let $T$ be a spanning forest of $G$. Then $T$ has $c(G)$ components and $n - c(G)$ edges. Furthermore, $T$ is bipartite and there exists a partition $V=V_1\sqcup V_2$ of the vertex set such that all edges of $T$ have one endpoint in $V_1$ and one endpoint in $V_2$. Hence, the number of edges between $V_1$ and $V_2$ in $G$, denoted $|E(V_1,V_2)|$, satisfies
\begin{equation*}
|E(V_1,V_2)|\geq n-c(G).
\end{equation*}
Now, let $f:V\rightarrow \mathbb{R}$ be defined by
\begin{equation*}
f(v):=\begin{cases}1, &\text{if }v\in V_1\\ -1,&\text{if }v\in V_2.\end{cases}
\end{equation*}
Then, by the Courant--Fischer--Weyl min-max Principle \cite[Thm. 1.2.1]{MHJ}, \cite[Thm. 2.4.1]{brouwer_2011_spectra},
\begin{equation*}
\lambda_{\max}(G)\geq \textrm{RQ}(f)=\frac{4\cdot |E(V_1,V_2)|}{2 m}\geq 2\cdot \frac{n-c(G)}{m} = 2\cdot \rho(G)^{-1}.
\end{equation*}
This proves \eqref{eq:lambdamax}. In the same way, one can show that
\begin{equation*}
\lambda_{\max}(G[A])\geq 2\cdot \rho(G[A])^{-1}=2\cdot \rho(A)^{-1}, \quad \text{for all } \emptyset\neq A\subseteq E.
\end{equation*}
Hence, if $G$ is uniformly dense, we can infer that
\begin{equation*}
\lambda_{\max}(G[A])\geq 2\cdot \rho(G)^{-1}, \quad \text{for all }\emptyset\neq A\subseteq E.
\end{equation*}This proves \eqref{eq:lambdamax-stable}.\newline

Similarly, let $\tilde{T}$ be a spanning forest of $G[A]$ for some $\emptyset \neq A\subseteq E$ and write $n_A=\vert V_A\vert$ and $m_A=\vert A\vert$. Then $\tilde{T}$ has $n_A-c(G[A])$ edges and, since it is bipartite, there exists a partition $V_A=V_1\sqcup V_2$ of its vertex set such that all edges of $\tilde{T}$ have one endpoint in $V_1$ and one endpoint in $V_2$. Hence, the number of edges between $V_1$ and $V_2$ in $G$ satisfies
\begin{equation*}
|E(V_1,V_2)|\geq n_A-c(G[A]).
\end{equation*}Also,
\begin{equation*}\label{eq:2delta}
\vol(V_1)+\vol(V_2)=\vol(V_A)\leq \vol_A(V_A)+2|\delta(A)|=2m_A+2|\delta(A)|.
\end{equation*}
Let now $f:V\rightarrow \mathbb{R}$ be defined by
\begin{equation*}
f(v):=\begin{cases}
1, &\text{if }v\in V_1,\\ -1,&\text{if }v\in V_2,\\ 0, &\text{otherwise.}
\end{cases}
\end{equation*}
Then, as in the proof of \cite[Theorem 4.5]{dual}, by the Courant--Fischer--Weyl min-max Principle,
\begin{equation*}
\lambda_{\max}(G)\geq \textrm{RQ}(f)\geq \frac{4\cdot|E(V_1,V_2)|}{\vol(V_1)+\vol(V_2)}\geq  \frac{2\cdot(n_A-c(G[A]))}{m_A+|\delta(A)|} = 2\cdot\rho(A)^{-1}\frac{m_A}{m_A + \vert\delta(A)\vert}.
\end{equation*} 
This proves \eqref{eq:A1}.
\end{proof}    

\begin{example}
Let $G$ be a tree graph. In this case $\rho(G)=1$. Also, $G$ is bipartite and the same holds for $G[A]$, for all $A\subseteq E$. Hence,
\begin{equation*}
        \lambda_{\max}(G)=\lambda_{\max}(G[A])=2,
\end{equation*}and both \eqref{eq:lambdamax} and \eqref{eq:lambdamax-stable} are equalities in this case. 
\end{example}

The next example gives a graph that does not satisfy \eqref{eq:lambdamax-stable} and therefore is not uniformly dense.

\begin{example}\label{example: clique and cycle}
Let $G$ be the graph given by the complete graph on $5$ vertices that shares a common edge with the cycle graph on $20$ vertices (see Figure \ref{fig: example clique and cycle}). 
Then, $c(G)=1$, $\vert V\vert =23$, and $\vert E\vert =29$. If $A$ is the set of edges of the complete graph on $5$ vertices, then
\begin{equation*} 
\lambda_{\max}(G[A]) = \frac{5}{4} < 2\cdot \frac{\vert V\vert-c(G)}{\vert E\vert} = \frac{44}{29}.
\end{equation*}
This shows that $G$ does not satisfy \eqref{eq:lambdamax-stable}. Hence, it is not uniformly dense.
\begin{figure}[h!]    
\centering    \includegraphics[width=0.2\textwidth]{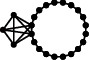}
\caption{The graph in Example \ref{example: clique and cycle}.}
\label{fig: example clique and cycle}
\end{figure}
\end{example}

\begin{remark}
    If $G$ is uniformly dense and $A\subseteq E$ is such that $G[A]$ is obtained from $G$ by simply deleting vertices and the edges in which these vertices are contained, then by Theorem \ref{thm:lambda-max} and by the Cauchy interlacing Theorem,
    \begin{equation*}
        \lambda_{\max}(G)\geq \lambda_{\max}(G[A])\geq 2\cdot \rho(G)^{-1}.
    \end{equation*}However, for a general set $A\subseteq E$, there are no interlacing results for $\lambda_{\max}(G)$ and $\lambda_{\max}(G[A])$.
\end{remark}

\begin{remark}
The last statement in Theorem \ref{thm:lambda-max} tells us that, the bigger the quantity $\rho(G)^{-1}$ is, and therefore the sparser $G$ is, then the bigger $\lambda_{\max}(G[A])$ must be, for all $A\subseteq E$. But, from classical results in spectral graph theory \cite{chung}, the bigger $\lambda_{\max}(G[A])$ is, then the sparser $G[A]$ must be, since the largest eigenvalue of the normalized Laplacian achieves its smallest possible value for complete graphs. This is in line with the interpretation of uniform density in Remark \ref{rmk:sparse}.
\end{remark}

For uniformly dense graphs, we can strengthen the lower bound \eqref{eq:lambdamax} on the largest Laplacian eigenvalue by making use of the following basis packing result of Catlin et al.\
\begin{lemma}[{\cite[Thm.\ 6]{catlin_1992_arboricity}}]\label{lemma: tree packing}
A uniformly dense graph contains $\bigl\lfloor \rho(G)\bigr\rfloor$ edge-disjoint spanning forests.
\end{lemma}
This leads to the following inequality:
\begin{theorem}
Let $G$ be a uniformly dense graph. Then
\begin{equation}\label{eq: lambdamax for stable graphs}
\lambda_{\max}(G) \geq 2\cdot \rho(G)^{-1} + 2\cdot \frac{\left(\left\lfloor\rho(G)\right\rfloor-1\right)c(G)}{\vert E\vert}.
\end{equation}
\end{theorem}
\begin{proof}
Let $G$ be a uniformly dense graph with $n$ vertices and $m$ edges. By Lemma \ref{lemma: tree packing}, there exists a set of $\bigl\lfloor \rho(G)\bigr\rfloor\geq 1$ edge-disjoint spanning forests. Let $T$ be any of these forests, with $V=V_1\sqcup V_2$ a vertex partition such that all edges of $T$ are between $V_1$ and $V_2$ and define $f:V\rightarrow\{-1,+1\}$ as in the proof of Theorem \ref{thm:lambda-max}. Then, by the Courant--Fischer--Weyl min-max Principle,
$$
\lambda_{\max}(G) \geq \RQ(f)  = \frac{4\cdot \vert E(V_1,V_2)\vert}{2m},
$$
where $\vert E(V_1,V_2)\vert$ is the number of edges between $V_1$ and $V_2$ in $G$. Since $T$ has all $n-c(G)$ edges between $V_1$ and $V_2$ and every other spanning forest has at least $c(G)$ edges between $V_1$ and $V_2$, we have
$\vert E(V_1,V_2)\vert \geq n-c(G) + (\lfloor \rho(G)\rfloor-1)c(G)$. This completes the proof.
\end{proof}

In general, \eqref{eq: lambdamax for stable graphs} is tighter than \eqref{eq:lambdamax} since $\lfloor \rho(G)\rfloor\geq 1$.
%%%%%%%%%%%%%%%%%%%%%%%%%%%%%%%%
%                              %
%%%%%%%%%%%%%%%%%%%%%%%%%%%%%%%%

\section{Real representable matroids}\label{section: real matroids}
We now consider real representable matroids, which are matroids whose structure is determined by a real matrix. We show that this additional algebraic structure guarantees that certain $E$-uniform basis measures with nice properties always exist for uniformly dense matroids (Theorem \ref{th: real UD matroids have determinantal measures}). As an application, we show that strictly uniformly dense real matroids can be represented by orthogonal projection matrices with constant diagonal (Theorem \ref{th: representable UD matroids and projection matrices}) and that they are parametrized by a subvariety of the Grassmannian. 
%%%%%%%%%%%%%%%%%%%%%%%%%
%                       %
%%%%%%%%%%%%%%%%%%%%%%%%%

\subsection{Definition and examples}
A matroid $M$ of size $n$ and rank $k$ is \emph{real representable} if there exists a real, full-rank $k\times n$ matrix $X$ with columns indexed by $E(M)$ such that a $k$-sized subset $B\subseteq E(M)$ is a basis if and only if $\det(X_B)\neq 0$, where $X_B$ is the submatrix with columns in $B$. We call $X=X(M)$ a \emph{representation} of $M$ and note that representations are not unique; in particular, scaling each column of $X$ by some nonzero number gives another representation. By the construction above, every full-rank $k\times n$ matrix $X$ determines a matroid $M(X)$. The linear-algebraic rank of a matrix $X$ and a subset of columns $X_A$ is equal to the corresponding rank in the matroid $M(X)$.

We give a number of examples of real representable matroids.
\begin{example}\label{ex: example of real representable matroid}
Let $M$ be the matroid with ground set $[5]$ and bases $\lbrace 123, 134, 234, 135, 345\rbrace$. This matroid is real representable and a matrix and vector representation of $M$ are shown in Figure \ref{fig: matrix and vector representation}.
\begin{figure}[h!]
    \centering    \includegraphics[width=0.6\textwidth]{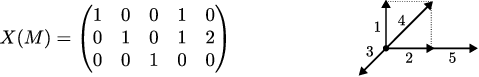}
    \caption{A matrix and vector representation of the matroid in Example \ref{ex: example of real representable matroid}.}
    \label{fig: matrix and vector representation}
\end{figure}
We note that scaling any of the columns of $X$ or, equivalently, any of the vectors by some positive real number $w$ still gives a valid representation. Since $3$ is a coloop and the other elements are not, $M$ is not uniformly dense.
\end{example}
\begin{example}[Incidence matrix]\label{example: incidence matrix}
Let $G=(V,E)$ be a graph and fix an arbitrary bijection $o_e:e=\lbrace u,v\rbrace\rightarrow \lbrace -1,1\rbrace$ for each $e\in E$. The {incidence matrix} of $G$ is the $\vert V(G)\vert\times \vert E(G)\vert$ matrix $\tilde{X}(G)$ with entries
$$
(\tilde{X}(G))_{ve} := \begin{cases}
o_e(v) \text{~if $v\in e$}\\
0 \text{~otherwise}.
\end{cases}
$$
The sum over all rows corresponding to the vertices of a connected component of $G$ is zero, which means that the incidence matrix is rank deficient by $c(G)$. The \emph{reduced incidence matrix} $X(G)$ is obtained by selecting one vertex in every connected component and deleting the corresponding row from the incidence matrix. This is a full-rank matrix that represents the graphic matroid of $G$, i.e.\ $M(X(G))=M(G)$ independent of the choice of $o_e$ and deleted vertices (i.e.\ rows); see for instance \cite{lyons_2003_determinantal, oxley_2011_matroid}. Figure \ref{fig: small graph and reduced incidence matrix} below shows the reduced incidence matrix from the graph in Example \ref{example: spanning trees of the 4 cycle with diagonal}.
\begin{figure}[h!]
    \centering
    \includegraphics[width=0.6\textwidth]{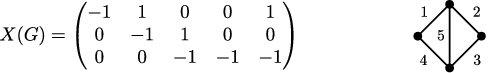}
    \caption{A graph $G$ and its reduced incidence matrix $X(G)$. The row corresponding to the leftmost vertex in $G$ is removed from the incidence matrix.}
    \label{fig: small graph and reduced incidence matrix}
\end{figure}
\end{example}
We consider a second, less standard representation of matroids. A real symmetric matrix $T$ is an orthogonal projection matrix if $T^2=T$. An orthogonal projection matrix $T=T(M)$ is a \emph{projection representation} of matroid $M$ if a $k$-sized subset $B\subseteq E(M)$ is a basis in $M$ if and only if the corresponding $k\times k$ principal submatrix $T_{BB}$ is non-singular. Projection representations are related to classical matrix representations by the identities $T=X^T(XX^T)^{-1}X$ and $\row(X)=\im(T)$, where $\row(~)$ denotes the row span. We also write $T=T(X)$ and $X=X(T)$ to make the relation between representations explicit.
\begin{example}
The projection representation of the matroid in Example \ref{example: spanning trees of the 4 cycle with diagonal} and \ref{example: incidence matrix} equals
$$
T = \frac{1}{8}\begin{pmatrix}
5&-1&-1&3&-2\\
-1&5&-3&1&2\\
-1&-3&5&1&2\\
3&1&1&5&2\\
-2&2&2&2&4
\end{pmatrix}.
$$
All $3\times 3$ principal minors of $T$ are nonzero, except $\det(T_{145,145})=\det(T_{235,235})=0$. These are precisely the non-bases of the matroid $M(T)$. For graphic matroids, this projection matrix is called the transfer current matrix \cite{lyons_2003_determinantal}, and it is a projection onto the cut-space of the graph.
\end{example}
%%%%%%%%%%%%%%%%%%%%%%%%%
%                       %
%%%%%%%%%%%%%%%%%%%%%%%%%

\subsection{Strictly uniformly dense real representable matroids}\label{subsection: SUD real representable}
In this section we show that strictly uniformly dense real representable matroids always have an $E$-uniform basis measure of a special form (see Theorem \ref{th: real UD matroids have determinantal measures}) and a projection representation with constant diagonal. From here on, we will abbreviate (strictly) uniformly dense real representable matroid by \emph{(strictly) uniformly dense real matroid}.

We will make use of the following result; see for instance \cite[Cor. 8.25]{michalek2021invitation} and \cite[Cor. 7.3.9]{sullivant2018algebraic}.
\begin{lemma}\label{lemma: Birch point}
Let the matrix $A\in\mathbb{Z}^{n\times m}$ have rank $n$ and $\mathbbm{1}\in\row(A)$, and let $p\in\mathbb{R}^m_{>0}$. If the linear system $Ay=b$ has a positive solution $y\in\mathbb{R}^m_{>0}$, then it has a unique positive solution $\hat{y}$ of the form $$\hat{y}_k = p_k\prod_{e=1}^n w_e^{a_{ek}},$$ for $k=1,\dots,m$, for some $w\in\mathbb{R}^n_{>0}$.
\end{lemma}
The point $\hat{y}$ is also called the \emph{Birch point}, and the prescribed form means that $\hat{y}$ lies on the positive real part of the toric variety parametrized by the integer matrix $A$. Lemma \ref{lemma: Birch point} is important in the context of maximum likelihood estimation for log-linear models in algebraic statistics \cite[Ch.\ 7.2]{sullivant2018algebraic}, and in the context of the so-called \emph{moment map}, which maps torus orbits in the Grassmannian to polytopes in Euclidean space, see \cite{gelfand_1987_combinatorial} and \cite[Ch.\ 8.2]{michalek2021invitation}. Here, we use Lemma \ref{lemma: Birch point} to prove the following characterization of strictly uniformly dense real matroids in terms of special kinds of measures. Recall that a matroid is strictly uniformly dense if $(\rho^{-1},\dots,\rho^{-1})$ is a relative interior point of the base polytope.

\begin{theorem}\label{th: real UD matroids have determinantal measures}
Let $M$ be a real representable matroid. Then $M$ is strictly uniformly dense if and only if it has a representation $X$ such that $\mu_X:\mathcal{B}(M)\ni B\mapsto \det(X_B)^2$ is an $E$-uniform basis measure.
\end{theorem}
\begin{proof}
The converse direction holds by Theorem \ref{th: main characterization theorem}. We prove the forward direction, starting with connected matroids. Let $M$ be a connected strictly uniformly dense real matroid of size $n$ and rank $k$, with $m:=\vert\mathcal{B}(M)\vert$ bases and with representation $X$. Define the $n\times m$ matrix $A$ with columns given by the indicator vectors $e_B$ of bases of $M$. This matrix has $\mathbbm{1}\in\row(A)$ since every column sums to $k=\rank(M)$. By connectivity and \cite[Prop. 2.4]{feichtner_2005_matroid} the dimension of the base polytope $P(M)\subset\mathbb{R}^n$ is $n-1$. Since $P(M)$ furthermore lies in a hyperplane away from the origin, its vertices (i.e., the columns of $A$) span $\mathbb{R}^n$ and thus $A$ has full rank. Define the $m\times 1$ vector $p$ with entries $p_B :=\det(X_B)^2>0$ for each $B\in\mathcal{B}(M)$. Since $M$ is strictly uniformly dense, there exists a normalized $E$-uniform measure $\mu$ on $\mathcal{B}$ with full support. Define the $m\times 1$ vector $y$ with entries $y_B:=\mu(B)>0$ for each $B\in\mathcal{B}(M)$. By definition of $E$-uniform measures and Lemma \ref{lemma: foster euler marginal sum}, we have $Ay=\rho^{-1}\cdot \mathbbm{1}$. 

Since $y$ is positive, Lemma \ref{lemma: Birch point} implies that there exists a unique positive solution to equation $A\hat{y} = \rho^{-1}\cdot \mathbbm{1}$ of the form $$\hat{y}_B = p_B\prod_{e=1}^nw_e^{a_{eB}} = \det(X_B)^2\prod_{e\in B}w_e$$ for some $w\in\mathbb{R}^n_{>0}$. Here we used that $a_{eB}=1$ if $e\in B$ and zero otherwise. Let $\tilde{X}$ be the matrix $X$ where the $e^{\text{th}}$ column is scaled by $\sqrt{w_e}$; this is still a representation of $M$ and it satisfies
$$
\det(\tilde{X}_B)^2 = \det(X_B)^2\prod_{e\in B}w_e.
$$ 
We can thus write $\hat{y}_B =\det(\tilde{X}_B)^2$. Now let $\mu_{\tilde{X}}(B) := \det(\tilde{X}_B)^2$, then by $A\hat{y}=\rho^{-1}\cdot \mathbbm{1}$ we know that $\mu_{\tilde{X}}$ is an $E$-uniform basis measure corresponding to the representation $\tilde{X}$. This completes the proof for connected matroids. The result for non-connected matroids $M=\bigoplus_{i=1}^\ell M_i$ follows directly from the fact that their representations decompose as $X(M)=[X(M_1)~\cdots~X(M_\ell)]$.
\end{proof}

We will call the measure in Theorem \ref{th: real UD matroids have determinantal measures} a \emph{determinantal measure} of $M$ and write $\mu_{X}$, where $X$ is the distinguished representation guaranteed by the theorem. This terminology refers to the fact that $\mu_X$ is called a \emph{projection determinantal measure} in the context of determinantal point processes, see for instance \cite{lyons_2003_determinantal}. 

In practice, the vector $w$ in Lemma \ref{lemma: Birch point} can be calculated using the so-called ``operator scaling'' algorithm \cite{garg_2018_algorithmic}. This iterative algorithm computes the appropriate scaling of the columns of a representation $X$, such that the corresponding determinantal measure $\mu_X$ in Theorem \ref{th: real UD matroids have determinantal measures} is $E$-uniform. This algorithm is another way to test if a real representable matroid is uniformly dense (see Corollary \ref{corollary: testing uniform density}). Moreover, the theory of operator scaling, for instance Theorem 1.2 in \cite{garg_2018_algorithmic}, guarantees that this approach is also efficient.

The second lemma we use is a linear-algebraic identity that follows from the theory of determinantal point processes; recall that we write $e_S$ for the indicator vector of a subset $S\subseteq [n]$.
\begin{lemma}\label{lemma: point in the polytope is diagonal of matrix}
Let $T$ be an orthogonal projection matrix of size $n$ and rank $k$. Then
\begin{equation}\label{eq: projection lemma}
\sum_{S\in{[n]\choose k}}\det(T_{SS})e_S= \diag(T).
\end{equation}
\end{lemma}
\begin{proof}
A projection determinantal point process corresponding to the orthogonal projection matrix $T$ is a normalized measure $\eta$ on $\binom{[n]}{k}$, that satisfies (see \cite[\S1--2]{lyons_2003_determinantal},\cite[\S1.1]{kassel_2022_determinantal}):
\begin{equation*}
\eta\big(\{S: S\supseteq\mathcal{I}\}\big) = \det(T_{\mathcal{I}\mathcal{I}}) \quad\text{~for all~}\quad\mathcal{I}\subseteq [n].
\end{equation*}
This implies $\eta(S)=\det(T_{SS})$ and the lemma follows by considering for each $i\in[n]$ the $i$th entry of equation \eqref{eq: projection lemma} and using the determinantal expression above for $\mathcal{I}=\{i\}$.
\end{proof}

Combining Theorem \ref{th: real UD matroids have determinantal measures} and Lemma \ref{lemma: point in the polytope is diagonal of matrix}, we find:

\projectionRepresentation

\begin{proof}
(Proof of forward direction.) Let $M$ be a strictly uniformly dense real matroid of size $n$ and rank $k$. By Theorem \ref{th: real UD matroids have determinantal measures}, $M$ has a representation $X$ such that the determinantal measure $\mu_X:\mathcal{B}\ni B\mapsto \det(X_B)^2$ is $E$-uniform. Let $T=T(X)=X^T(XX^T)^{-1}X$ be the corresponding projection representation and observe, by multiplicativity of the determinant, that 
$$
\det(T_{SS})=\det(X_S)\det(XX^T)^{-1}\det(X_S)=\frac{\det(X_S)^2}{\det(XX^T)}, \quad \text{for all }S\in{[n]\choose k}.
$$ 
By the Cauchy--Binet theorem \cite[\S2.5]{kassel_2022_determinantal} we have $
\det(XX^T) =  \sum_{S\in\binom{[n]}{k}}\det(X_S)^2
$, and thus
$$
\det(T_{SS}) = \frac{\det(X_S)^2}{\sum_{S'\in\binom{[n]}{k}}\det(X_{S'})^2},\quad\text{~for all $S\in\binom{[n]}{k}$}.
$$
We can thus write the normalized determinantal basis measure as $
\tilde{\mu}_X:\mathcal{B}\ni B\mapsto \det(T_{BB}).
$ By Lemma \ref{lemma: point in the polytope is diagonal of matrix} we then have
$$
\diag(T) = \sum_{S\in{[n]\choose k}}\det(T_{SS})e_S = \sum_{B\in\mathcal{B}(M)}\tfrac{1}{Z}\det(X_B)^2 e_B,
$$
where $Z=\sum_{B\in\mathcal{B}(M)}\det(X_B)^2$. Looking at the $e^{\text{th}}$ diagonal entry, we find
$$
T_{ee} = \sum_{B\ni e}\tfrac{1}{Z}\det(X_B)^2 = \rho(M)^{-1},
$$
where the second equality follows from Lemma \ref{lemma: foster euler marginal sum}. Thus $T$ has constant diagonal equal to $\rho(M)^{-1}$ which completes the proof of the forward direction.

(Proof of converse direction.) Let $T$ be an orthogonal projection matrix with constant diagonal $k/n$ and let $\tilde{X}=\tilde{X}(T)$. Then $T$ is a projection representation of matroid $M(\tilde{X})$ with density $\rho(M(\tilde{X}))=n/k$. By the same steps as in the forward direction, the normalized measure 
$$
\mu_{T}(B) := \det(T_{BB}) = \tfrac{1}{Z}\det(\tilde{X}_{B})^2 = \mu_{\tilde{X}}(B)
$$ 
is an $E$-uniform measure on $\mathcal{B}(M(\tilde{X}))$ and thus $M(\tilde{X})$ is strictly uniformly dense.
\end{proof}

\begin{example}\label{example: finding determinantal measure} 
Let $M$ be the matroid from Example \ref{example: spanning trees of the 4 cycle with diagonal}. This is the rank $3$ matroid on $[5]$ with non-bases $125$ and $345$. This matroid is graphic, real representable and strictly uniformly dense. Starting from the representation in Figure \ref{fig: small graph and reduced incidence matrix}, i.e.\ the reduced incidence matrix of the associated graph, we use the operator scaling algorithm to calculate the column scaling vector $w=(2,2,2,2,3)$ appearing in Lemma \ref{lemma: Birch point} and the proof of Theorem \ref{th: real UD matroids have determinantal measures}. This gives the representation $$\tilde{X}=\left(\begin{smallmatrix}-\sqrt{2}&\sqrt{2}&0&0&\sqrt{3}\\0&-\sqrt{2}&\sqrt{2}&0&0\\0&0&-\sqrt{2}&-\sqrt{2}&-\sqrt{3} \end{smallmatrix}\right)$$ whose corresponding determinantal measure has $\mu_{\tilde{X}}(B)=12$ for all bases that contain element $5$ in the ground set (the diagonal edge in the corresponding graph) and $\mu_{\tilde{X}}(B)=8$ otherwise. This is an $E$-uniform basis measure for $M$. The projection matrix onto $\row(\tilde{X})$ equals
$$ 
T(\tilde{X}) = \frac{1}{10}\begin{pmatrix}
6&-1&-1&4&-\sqrt{6}\\
-1&6&-4&1&\sqrt{6}\\
-1&-4&6&1&\sqrt{6}\\
4&1&1&6&\sqrt{6}\\
-\sqrt{6}&\sqrt{6}&\sqrt{6}&\sqrt{6}&6
\end{pmatrix}.
$$
This matrix has constant diagonal $3/5$ and its non-singular $3\times 3$ principal submatrices correspond to the bases of $M$. This is the projection representation from Theorem \ref{th: representable UD matroids and projection matrices}.
\end{example}

Finally, we note an application of Theorem \ref{th: representable UD matroids and projection matrices} to the properties of projection matrices.
\begin{corollary}
Let $T$ be a constant-diagonal orthogonal projection matrix of size $n$ and rank $k$. Then for all $0<m<n$, every principal $m\times m$ submatrix $T'$ of $T$ has $\rank(T')\geq m\cdot k/n$.
\end{corollary}
\begin{proof}
Let $T$ be a constant-diagonal orthogonal projection matrix of size $n$ and rank $k$. Following Theorem \ref{th: representable UD matroids and projection matrices}, let $M=M(T)$ be the strictly uniformly dense real matroid represented by $T$ and $X=X(T)$ the corresponding standard representation. Note that $T_{AA}=X_{A}^T(XX^T)^{-1}X_{A}$ for all $A\subseteq [n]$, which implies
$$
\rank(T_{AA})=\rank(X_A)=\rank_{M}(A) \geq \rank_M(E)\frac{\vert A\vert}{\vert E\vert}=k\cdot \vert A\vert/n \text{~for all $\emptyset\neq A\subset E(M)$}.
$$
Here, the first two ranks are linear-algebraic, while the third and fourth rank$_M$ are the rank function in the matroid $M$. The inequality follows from uniform density of $M$.
\end{proof}
%%%%%%%%%%%%%%%%%%%%%%%%%
%                       %
%%%%%%%%%%%%%%%%%%%%%%%%%

\subsection{The variety of uniformly dense real matroids}\label{subsection: moduli space}
As discussed earlier, a full-rank $k\times n$ matrix $X$ determines a real representable matroid $M(X)$. The set of all such matrices modulo elementary row operations is called the real \emph{Grassmannian} $\operatorname{Gr}(n,k)$; equivalently, this is the set of $k$-dimensional subspaces of $\mathbb{R}^n$. Points in the Grassmannian parametrize the real representable matroids. The embedding $p:\Gr(n,k)\hookrightarrow \mathbb{P}^{{n\choose k}-1}$ maps a matrix $X\in\Gr(n,k)$ onto its $k$-minors $p_B(X) := \det(X_B)$ for each $B\in{[n]\choose k}$, up to a common non-zero scaling of all minors. These minors $p_B(X)$ are called the \emph{Pl\"{u}cker coordinates} of $X$, and the vanishing of certain quadratic polynomials in these coordinates (the Pl\"{u}cker relations) defines the Grassmannian as a variety in $\mathbb{P}^{{n\choose k}-1}$; see for instance \cite[Ch.\ 5]{michalek2021invitation} for more details.

Following Theorem \ref{th: real UD matroids have determinantal measures}, every strictly uniformly dense real matroid $M$ has a distinguished representation $X$ whose Pl\"{u}cker coordinates ($p_B(X)=\sqrt{\mu_X(B)}$) satisfy the quadratic relations $$\sum_{B\ni e}p^2_B(X)=\sum_{B\ni e'}p^2_B(X),\quad \text{for all }e\neq e'\in E(M).$$ This means that every strictly uniformly dense real matroid has a representation in the following projective subvariety of the Grassmannian:
$$
\mathcal{V}(n,k) := \left\lbrace p\in\Gr(n,k) \,:\, \sum_{e\in B\in{[n]\choose k}} p_B^2 - \sum_{(e+1)\in B\in{[n]\choose k}}p_B^2=0\text{~for $e=1,\dots,n-1$} \right\rbrace \subseteq\mathbb{P}^{{n\choose k}-1}.
$$
Conversely, every point in $\mathcal{V}(n,k)$ corresponds to a matrix $X$ whose non-zero Pl\"{u}cker coordinates determine the bases of a strictly uniformly dense real matroid $M(X)$. In other words, the variety $\mathcal{V}(n,k)$ parametrizes the strictly uniformly dense real matroids.

A different description follows from Theorem \ref{th: representable UD matroids and projection matrices}. Let $T$ be an $n\times n$ symmetric matrix with constant diagonal $k/n$. This is a projection matrix if and only if it satisfies $T^2 - T = 0$; in other words, it is defined by the vanishing of ${n+1\choose 2}$ quadratic polynomials in its ${n\choose 2}$ off-diagonal entries $\lbrace t_{ij}\rbrace_{1\leq i<j\leq n}$. Another description of $\mathcal{V}(n,k)$ is thus as the affine variety
$$
\mathcal{V}(n,k) = \left\lbrace t\in\mathbb{R}^{{n\choose 2}} \,:\, T^2-T = 0\right\rbrace,
$$
where $T$ is the symmetric matrix with constant diagonal $k/n$ and off-diagonal entries $t_{ij}$. Every point in $\mathcal{V}(n,k)$ is the projection representation of a strictly uniformly dense real matroid. The algebro-geometric properties of the variety $\mathcal{V}(n,k)$ in its projective and affine embedding are further studied in \cite{devriendt_2024_grassmannian}.

%%%%%%%%%%%%%%%%%%%%%%
%  Acknowledgements  %
%%%%%%%%%%%%%%%%%%%%%%
\section*{Acknowledgements}
We are grateful to Bernd Sturmfels, Benjamin Schr\"{o}ter and Harry Richman for helpful comments, discussions and suggestions. Raffaella Mulas is supported by the Dutch Research Council (NWO) through the grant VI.Veni.232.002.
%\section*{Funding}

%%%%%%%%%%%%%%%%%
%  BIBLIOGRAPHY %
%%%%%%%%%%%%%%%%%
%\thebibliography

\end{document}